\documentclass[reqno]{amsart}

\usepackage[colorlinks=true,breaklinks=true,bookmarks=true,urlcolor=blue,
citecolor=blue,linkcolor=blue,bookmarksopen=false,draft=false]{hyperref}

\usepackage{tikz}
\usepackage{verbatim}
\usepackage{color}
\usepackage{amsmath}
\usepackage{amssymb}
\usepackage{amsthm}
\definecolor{MyDarkGreen}{rgb}{0,0.45,0}
\definecolor{MyDarkRed}{rgb}{0.9,0,0}

\newtheorem{theorem}{Theorem}
\newtheorem{lemma}{Lemma}

\newtheorem{corollary}{Corollary}

\newtheorem{example}[lemma]{Example}
\newtheorem{remark}[lemma]{Remark}

\newcommand{\E}{{{\bf E}}}
\newcommand{\eqdef}{\mathrel{:=}}
\newcommand{\myvec}[1]{\mathbf{#1}}
\newcommand{\eps}{\varepsilon}
\newcommand{\NN}{\mathbb N}

\newcommand{\vertiii}[1]{{\left\vert\mkern-1.5mu\left\vert\mkern-1.5mu\left\vert #1 
		\right\vert\mkern-1.5mu\right\vert\mkern-1.5mu\right\vert}}
\newcommand{\altvertiii}[1]{{\big\vert\mkern-1.5mu\big\vert\mkern-1.5mu\big\vert #1 
		\big\vert\mkern-1.5mu\big\vert\mkern-1.5mu\big\vert}}
\newcommand{\tddt}{\tfrac{d \hfill}{dt}}
\title[Comp. limit of discrete Bass models]{Compartmental limit of discrete Bass models on networks}
\author[G. Fibich, A. Golan, and S. Schochet]{Gadi Fibich, Amit Golan, and Steve Schochet\\
	\tiny{Department of Applied Mathematics, Tel Aviv University (\href{mailto:fibich@tau.ac.il}{fibich@tau.ac.il}), (\href{mailto:amitgolan33@gmail.com}{amitgolan33@gmail.com}),
		(\href{mailto:schochet@tauex.tau.ac.il}{schochet@tauex.tau.ac.il}).}}
\thanks{		
	MSC[2020] Primary: 91D30, Secondary: 34E10, 60J80, 92D25}

\begin{document}
	\begin{abstract}
		We introduce a new method for proving the convergence and the rate of convergence of discrete Bass models on various networks to their respective compartmental Bass models, as the population size~$M$ becomes infinite.
		In this method, the full set of master equations
		is reduced to a smaller system of equations, which is closed and exact. 
		The reduced finite system is embedded into an infinite system, 
		and the convergence of that system to the infinite limit system
		is proved using standard ODE estimates.	Finally, an ansatz provides an exact closure of the infinite limit system, which reduces that system to the compartmental model.
		
		Using this method, we show that when the network is complete and homogeneous, the discrete Bass model converges to the original 1969 compartmental Bass model, at the rate of $1/M$. When the network is circular, however, the compartmental limit is different, and the rate of convergence is exponential in $M$. In the case of a heterogeneous network that
		consists of $K$ homogeneous groups, the limit is given by a heterogeneous compartmental Bass model, and the rate of convergence is $1/M$. Using this compartmental model, we show that when the heterogeneity in the external
		and internal influence parameters among the $K$ groups is positively  monotonically related, heterogeneity slows down the diffusion.
	\end{abstract}
	
	\maketitle
	\section{Introduction}
	\label{sec:introduction}
	Diffusion of innovations in networks has attracted the attention of researchers  
	in physics, mathematics, biology, computer science, social sciences, economics, and management science,
	as it concerns the spreading of ``items'' ranging from diseases and computer viruses to rumors, information, opinions, technologies, and 
	innovations~\cite{Albert-00,Anderson-92,Jackson-08,Pastor-Satorras-01,Rogers-03,Strang-98}. 
	In marketing, diffusion of new products plays a key role, with applications in retail service, industrial technology, agriculture, and in educational, pharmaceutical, and consumer-durables markets \cite{Mahajan-93}.

	The first quantitative model of the diffusion of new products was proposed in 1969 by 
	Bass~\cite{Bass-69}. 
	In this model, the rate of change of the number of individuals who adopted the product by time $t$ is 
	\begin{equation}
		\label{eq:Bass_intro}
		n'(t)= \left(M-n\right)\left(p+\frac{q}{M}n\right), \qquad n(0)=0,
	\end{equation}
	where $n$ is the number of adopters, $M$ is the population size, $M-n$ are the remaining potential adopters, $p$ is the rate of external influences by mass media (TV, newspapers,...) on any nonadopter to adopt, and $\frac{q}{M}$ is the rate of internal influences by any adopter on any nonadopter to adopt (``word of mouth", ``peer effect"). Internal influences are additive, so that the overall rate of internal influences is proportional to $n$.
	
	The Bass model~\eqref{eq:Bass_intro} is a {\em compartmental model}. Thus, the population is divided into two {\em compartments} (groups), adopters and nonadopters, and individuals move between these two compartments at the rate given by~\eqref{eq:Bass_intro}. The Bass model is one of the most cited papers in \textit{Management Science}~\cite{Hopp-04}. Almost all its extensions have also been compartmental models;  
	given by a deterministic ODE or ODEs. The main advantage of compartmental models is that they are easy to analyze. {\em From a modeling perspective, however, one should start from first principles, and model the adoption of each individual using a stochastic ``particle model".} The macroscopic/aggregate dynamics should then be {\em derived} from this discrete Bass model, rather than assumed phenomenologically, which is done in compartmental Bass models. Moreover, the discrete Bass model allows us to relax the assumption that all individuals are connected (i.e., that the social network is a ``complete graph") and have {\em any network structure}. The discrete Bass model also enables us to relax the assumption that individuals are homogeneous, and allows for {\em heterogeneous individuals}, which is much more realistic.
	
	At present, the only rigorous result on the relation between discrete and compartmental Bass models is by Niu~\cite{Niu-02}, who {\em derived} the compartmental Bass model~\eqref{eq:Bass_intro} as the $M\to\infty$ limit of the discrete Bass model on a homogeneous complete network. The approach in~\cite{Niu-02}, however, does not extend to other types of networks, nor does it provide the rate of convergence. Fibich and Gibori derived an explicit expression for the macroscopic diffusion in the discrete Bass model on infinite circles~\cite{OR-10}. They did not prove rigorously, however, that this expression is the limit of the discrete Bass model on a circle with $M$ nodes as $M\to\infty$, nor did they find the rate of convergence.
	
	In this paper, we present a novel method for proving the convergence of discrete Bass models.{\em This method can be applied to various network types, and it also provides the convergence rate}. Since real networks are finite, the convergence rate provides an estimate for the difference between a finite network and its infinite-population compartmental limit.
	
	We first use our method to provide an alternative proof to the convergence of the discrete Bass model on a homogeneous complete network, and to show that the rate of convergence is $\frac{1}{M}$. We then use this method to prove the convergence of the discrete Bass model on the infinite circle, and to show that the rate of convergence is exponential in $M$. Finally, we use this method to prove the convergence of a discrete Bass model in a
	heterogeneous network. Specifically, we consider a heterogeneous population which consists of $K$ groups, each of which
	is homogeneous. We show that as $M\to\infty$, the fraction of adopters in the heterogeneous discrete model approaches that of the
	compartmental model. Then, we analyze the qualitative effect of heterogeneity in the {\em heterogeneous compartmental
		Bass model}. In particular, we show that when the heterogeneity is just in $\{p_j\}$, just in $\{q_j\}$, or when
	$\{p_j\}$ and $\{q_j\}$ are positively  monotonically related, then heterogeneity slows down the diffusion.
	
	The main contributions of this paper are:
	\begin{enumerate}
		\item
		A new method for proving the convergence and the rate of convergence of discrete Bass models as $M\to\infty$. This method is based on embedding a system of ODEs with a varying number of equations in an infinite system.
		\item
		A convergence proof for the discrete Bass model on the circle, and for a heterogeneous network with $K$ groups.
		\item
		Finding the rate of convergence of the discrete Bass model on a homogeneous complete network, a heterogeneous network with $K$ groups, and a homogeneous circle.
		\item
		An elementary proof that heterogeneity slows down the diffusion whenever the heterogeneity is just in $\{p_j\}$, just
		in $\{q_j\}$, or when $\{p_j\}$ and $\{q_j\}$ are positively  monotonically related.
	\end{enumerate}
	\section{Discrete Bass model}
	\label{sec:discrete_Bass}
	We begin by introducing the discrete Bass model for the diffusion of new products. A new product is introduced at time $t=0$ to a network with $M$ consumers. We denote by $X_j(t)$ the state of consumer $j$ at time $t$, so that 
	\begin{equation*}
		X_j(t)=\begin{cases}
			1, \qquad {\rm if}\ j\ {\rm adopts\ the\ product\ by\ time}\ t,\\
			0, \qquad {\rm otherwise.}
		\end{cases}
	\end{equation*} 
	Since all consumers are initially nonadopters,
	\begin{equation}
		\label{eq:general_initial}
		X_j(0)=0, \qquad j=1,\dots,M.
	\end{equation}
	Once a consumer adopts the product, she remains an adopter for all time. The underlying social network is represented by
	a weighted directed graph, where the weight of the edge from node $i$ to node $j$ is $q_{i,j}\geq 0$, and $q_{i,j}=0$ if
	there is no edge from $i$ to $j$. We scale the weights so that
	if $i$ already adopted the product and $q_{i,j}>0$, her rate of internal influence on consumer $j$ to adopt is
	$\frac{q_{i,j}}{{d_j(M)}}$, where $d_j(M)$ is the number of edges leading to node~$j$
	(the {\em indegree} of node $j$).
	This scaling ensures that the maximal internal influence
	\begin{equation}
		\label{eq:q_on_node}
		q_j:=\sum_{\substack{m=1\\m\neq j}}^{M}\frac{q_{m,j}}{{d_j(M)}}
	\end{equation}  
	on a nonadopter, which occurs when all her peers are adopters, will remain bounded as $M$ tends to infinity if the $q_{m,j}$ are bounded, and will
	equal their common value when all the $q_{m,j}$ corresponding to edges leading to $j$ are equal.
	In addition, consumer $j$ experiences an external influence to adopt, at the rate of $p_j> 0$. Hence, to first order in $\Delta t$,
	\begin{equation}
		\label{eq:discrete_model}
		{\rm Prob}(X_j(t+\Delta t)=1\ |\ {\bf X}(t))=\begin{cases}
			\hfill 1,\hfill &\qquad {\rm if}\ X_j(t)=1,\\
			\left(p_j+\sum\limits_{\substack{m=1 \\m\neq j}}^M\frac{q_{m,j}}{{d_j(M)}}X_{m}(t)\right)\Delta t, &\qquad {\rm if}\ X_j(t)=0,
		\end{cases}
	\end{equation}
	where ${\bf X}(t):=\left(X_1(t),\ldots,X_M(t)\right)$ is the state of the network at time $t$.
	\newsavebox{\smlmat}
	\savebox{\smlmat}{$\left(q_{i,j}\right)=
		\left(\begin{smallmatrix}
			0.1&0.05&0.01&0\\
			0.05&0.025&0.08&0.05\\
			0.01&0.02&0.03&0.04\\
			0.15&0.05&0.05&0.05
		\end{smallmatrix}\right)$ }
	The quantity of most interest is the expected fraction of adopters
	\begin{equation}
		\label{eq:expected_frac}
		f_{\rm discrete}(t;\{p_j\},\{q_{m,j}\},M)=\frac{1}{M}\E\left[N(t)\right],
	\end{equation}
	where $N(t):=\sum_{j=1}^{M}X_j(t)$ is the number of adopters at time $t$.
	
	Let $[S_{m_1},\dots,S_{m_n}]:={\rm Prob}(X_{m_i}=0,\ i=1,\dots,n)$ denote the probability that the $n$ nodes $\{m_1,\dots,m_n\}$ are nonadopters, where $1\leq n\leq M$, $m_i\in \{1,\dots,M\}$, and $m_i\neq m_j$ if $i\neq j$. These probabilities satisfy the master equations:
	\begin{lemma}[{\cite{fibich2021diffusion}}]
		\label{lem:mastereqs}
		The master equations for the discrete Bass model~$(\ref{eq:general_initial},\ref{eq:discrete_model})$ are
		\begin{subequations}
			\label{eqs:masterHeterofc}
			\begin{equation}
				\label{eq:masterHeterofc}
				\begin{aligned}
					\frac{d}{dt}[S_{m_1},\dots,S_{m_n}](t)=
					&-\left(\sum_{i=1}^{n}p_{m_i}+\sum_{j=n+1}^{M}\sum_{i=1}^{n}
					\frac{q_{l_j,m_i}}{{d_j(M)}}\right)[S_{m_1},\dots,S_{m_n}]
					\\&+\sum_{j=n+1}^{M}\left(\sum_{i=1}^{n} \frac{q_{l_j,m_i}}{{d_j(M)}}\right)[S_{m_1},\dots,S_{m_n},S_{l_j}],
				\end{aligned}
			\end{equation}
			for $\{m_1,\dots,m_n\} \subsetneq \{1,\dots,M\},$
			where $\{l_{n+1},\dots,l_M\}=\{1,\dots,M\}\backslash\{m_1,\dots,m_n\}$,  and
			\begin{equation}
				\label{eq:masterHeterofcM}
				\frac{d}{dt}[S_{1},\dots,S_{M}](t)=-\left(\sum_{i=1}^M p_i\right)[S_{1},\dots,S_{M}],
			\end{equation}
			subject to the initial conditions
			\begin{equation}
				\label{eq:masterHeterofcInitial}
				[S_{m_1},\dots,S_{m_n}](0)=1, \qquad\{m_1,\dots,m_n\}{\subseteq} \{1,\dots,M\}.
			\end{equation}
		\end{subequations}
	\end{lemma}
	In what follows, we will use equations~\eqref{eqs:masterHeterofc} to analyze the limit of $f_{\rm discrete}$ as $M\to\infty$. Indeed, since $\E\left[X_j(t)\right]=1-[S_j](t)$, then
	\begin{equation}
		\label{eq:f_discrete}
		f_{\rm discrete}=1-\frac{1}{M}\sum_{j=1}^{M}[S_j].
	\end{equation}
	Therefore,
	\begin{equation*}
		\lim\limits_{M\to\infty}f_{\rm discrete}(\cdot,M)=1-\lim\limits_{M\to\infty}\frac{1}{M}\sum_{j=1}^{M}[S_j](\cdot,M).
	\end{equation*}
	
	\subsection{Relation between discrete and compartmental Bass models}
	\label{subsec:compartmental_relation}
	From a modeling perspective, the discrete Bass model is more fundamental than the compartmental model. The latter model, however, is much easier to analyze. 
	Indeed, the homogeneous compartmental Bass model~\eqref{eq:Bass_intro} can be rewritten as
	\begin{equation}
		\label{eq:homogeneous_Bass}
		\frac{d}{dt}f(t)=(1-f)(p+qf),\qquad f(0)=0,
	\end{equation}
	where $f:=\frac{n}{M}$ is the fraction of adopters. This equation can be easily solved, yielding the Bass formula~\cite{Bass-69}
	\begin{equation}
		\label{eq:Bass_sol}
		f_{\rm Bass}(t;p,q)=\frac{1-e^{-(p+q)t}}{1+\frac{q}{p}e^{-(p+q)t}}.
	\end{equation}
	The corresponding discrete network is complete and homogeneous, i.e.
	\begin{equation}
		\label{eq:hom_conditions}
		p_j\equiv p,\qquad q_{k,j}\equiv q,\qquad d_j(M)\equiv M-1,\qquad k,j=1,\dots,M, \qquad k\neq j.
	\end{equation} 
	In that case, \eqref{eq:discrete_model}~reads
	\begin{equation}
		\label{eq:general_homogeneous}
		{\rm Prob}(X_j(t+\Delta t)=1\ |\ {\bf X}(t))=\begin{cases}
			\hfill 1, \hfill & \qquad {\rm if}\ X_j(t)=1,\\
			\left(p+\frac{q}{M-1}N(t)\right)\Delta t, & \qquad {\rm if}\ X_j(t)=0.
		\end{cases}
	\end{equation}
	The relation between the discrete Bass model on a homogeneous network and the compartmental Bass model was established by Niu:
	\begin{theorem} [{\cite{Niu-02}}]
		\label{thm:niu}
		Let $p,q>0$. Then the expected fraction of adopters in the discrete Bass model~$(\ref{eq:general_initial},\ref{eq:general_homogeneous})$ on a homogeneous complete network approaches that of the homogeneous compartmental Bass model~\eqref{eq:homogeneous_Bass} as $M\to\infty$, i.e.,
		\begin{equation}
			\label{eq:f_Bass-Niu}
			\lim\limits_{M\to\infty}f^{\rm complete}_{\rm discrete}(t;p,q,M)=f_{\rm Bass}(t;p,q).
		\end{equation}
	\end{theorem}
	As far as we know, Theorem~\ref{thm:niu} is the only previous rigorous proof of convergence of any discrete Bass model as $M\to\infty$.
	
	\section{Homogeneous complete network}
	\label{sec:hom_comp}
	In this section we introduce a novel method for proving Theorem~\ref{thm:niu}. This method also provides the {\em rate of convergence}, and can be extended to other types of networks.
	
	\begin{theorem}
		\label{thm:Niu}
		Assume the conditions of Theorem~\ref{thm:niu}. Then the limit~\eqref{eq:f_Bass-Niu} is uniform in~$t$. Moreover, the rate of convergence is $\frac1M$, i.e.,
		\begin{equation}
			\label{eq:f_Bass-rate-of-convergence-hom}
			f^{\rm complete}_{\rm discrete}(t;p,q,M) - f_{\rm Bass}(t;p,q) = O\left(\frac1M\right), \qquad M \to \infty.
		\end{equation}
	\end{theorem}
	\begin{proof}
		Our starting point are the master equations~\eqref{eqs:masterHeterofc}.
		When the network is homogeneous and complete, see~\eqref{eq:hom_conditions}, then by symmetry, $[S_{m_1},\dots,S_{m_n}]\\=[S_{k_1},\dots,S_{k_n}]$ for any $\{m_1,\dots,m_n\}$ and $\{k_1,\dots,k_n\}\subset\{1,\dots,M\}$. Hence, we can denote by $[S^n](t)$ the probability that any arbitrary subset of $n$ nodes are nonadopters at time~$t$. Using this symmetry and~\eqref{eq:hom_conditions}, the master equations~\eqref{eqs:masterHeterofc} reduce to
		\begin{subequations}
			\label{eqs:masterfc}
			\begin{equation}
				\label{eq:masterfc}
				\frac{d}{dt}[S^n](t;M)=-n\left(p+\frac{M-n}{M-1}q\right)[S^n]+n\frac{M-n}{M-1}q[S^{n+1}], \qquad n=1,\dots,M-1,
			\end{equation}
			\begin{equation}
				\label{eq:masterfcM}
				\frac{d}{dt}[S^M](t;M)=-Mp[S^M],
			\end{equation}
			subject to the initial conditions
			\begin{equation}
				\label{eq:masterfcInitial}
				[S^n](0;M)=1, \qquad n=1,\dots,M.
			\end{equation}
		\end{subequations}
		Moreover, by~\eqref{eq:f_discrete}, 
		\begin{equation}
			\label{eq:f_disc_simple}
			f^{\rm complete}_{\rm discrete}=1-[S], \qquad [S]:=[S^1].
		\end{equation}
		If we formally fix~$n$ and let $M\to\infty$ in~\eqref{eqs:masterfc}, we get that  
		\begin{equation}
			\label{eq:master_inf-hom}
			\frac{d}{dt}[S^n_{\infty}](t)=-n(p+q)[S^n_{\infty}]+nq[S^{n+1}_{\infty}], \qquad [S^n_{\infty}](0) = 1, \qquad n=1,2,\dots.
		\end{equation}
		This does not immediately imply that $\lim_{M \to \infty} [S^n] =  [S^n_{\infty}]$,
		since the number of ODEs in~\eqref{eqs:masterfc} increases with~$M$, and becomes infinite in the limit. In Lemma~\ref{lem:Steve-convergence-homog} below, however,
		we will prove that for any $n \ge 1$,
		\begin{subequations}
			\label{eqs:s^M->S}
			\begin{equation}
				\label{eq:S^M->S-complete-hom}
				\lim_{M \to \infty} [S^n](t;M)  =  [S^n_{\infty}](t), \qquad \mbox{uniformly in $t$}.
			\end{equation}
			Moreover, 
			\begin{equation}
				\label{eq:S^M->S-complete-rate}
				[S^n](t;M)  - [S^n_{\infty}](t)  = O\left(\frac1M\right), \qquad M \to \infty.
			\end{equation}
		\end{subequations}
		Therefore,  
		we can proceed to solve the infinite system~\eqref{eq:master_inf-hom}.
		
		To do that, we note that substituting the ansatz 
		\begin{equation}
			\label{eq:inf_sol_hom}
			[S^n_{\infty}]=  [S_{\infty}]^n,\qquad n=1,2,\dots
		\end{equation}
		into \eqref{eq:master_inf-hom} transforms that system into
		\begin{equation*}
			n[S_{\infty}]^{n-1}\frac{d}{dt}[S_{\infty}](t)=-n(p+q)[S_{\infty}]^n+nq[S_{\infty}]^{n+1},
			\qquad [S_{\infty}](0)=1.
		\end{equation*}
		Dividing by $n[S_{\infty}]^{n-1}$, we find that the infinite system reduces to the single ODE and initial condition 
		\begin{equation*}
			\frac{d}{dt}[S_{\infty}]=-\left(p+q\right)[S_{\infty}]+  q[S_{\infty}]^2
			=-[S_{\infty}] (p+q(1-[S_{\infty}]), \qquad [S_{\infty}](0)=1.
		\end{equation*}
		Therefore, by~\eqref{eq:homogeneous_Bass},
		\begin{equation}
			\label{eq:s_infty_hom}
			[S_{\infty}] = 1-f_{\rm Bass}.
		\end{equation}
		The results follow
		from~\eqref{eq:f_disc_simple},~\eqref{eq:inf_sol_hom},~\eqref{eq:s_infty_hom}, and~\eqref{eq:S^M->S-complete-hom} and~\eqref{eq:S^M->S-complete-rate} with $n=1$.
	\end{proof}
	The $O\left(\frac1M\right)$ rate of convergence which is predicted in Theorem~\ref{thm:Niu} is illustrated numerically in Figure~\ref{fig:convergence_hom}, where $f_{\rm Bass}-f^{\rm complete}_{\rm discrete}\approx\frac{C}{M^{0.99}}$. Here, $f^{\rm complete}_{\rm discrete}$ was calculated from the average of $10^6$ simulations of~$(\ref{eq:general_initial},\ref{eq:general_homogeneous})$.
	
	\begin{figure}[ht!]
		\begin{center}
			\scalebox{1}{\includegraphics{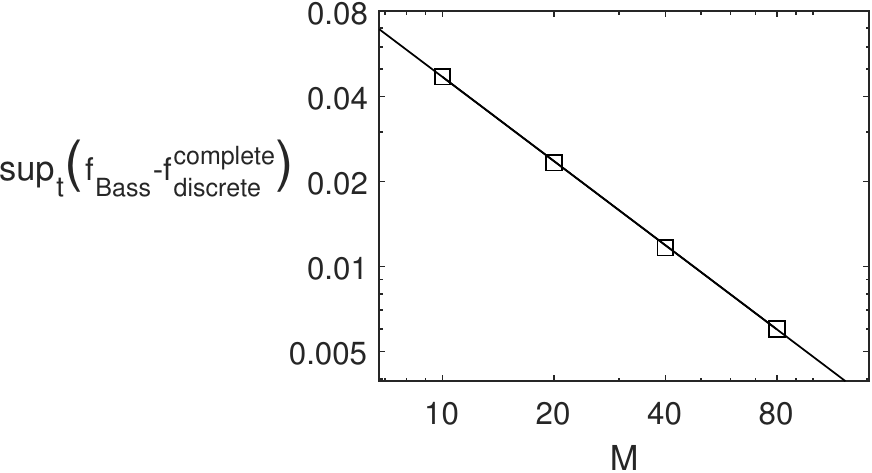}}
			\caption{Log-log plot of $\sup_{t}\left(f_{\rm Bass}-f_{\rm discrete}^{\rm complete}\right)$ as a function of $M$. The fitted solid line is $\log y=-0.99\log M-0.78$. Here, $p=0.02$ and $q=0.1$.}
			\label{fig:convergence_hom}
		\end{center}
	\end{figure}
	
	\subsection{Convergence and rate of convergence}
	\label{sec:proof-Niu}
	The proof of Theorem~\ref{thm:Niu} makes use of
	\begin{lemma}
		\label{lem:Steve-convergence-homog}
		For any $n \in \mathbb{N}$ and $p,q>0$, the solution $[S^n](t;p,q,M)$  of~\eqref{eqs:masterfc}
		converges, uniformly in~$t$, to the solution $[S^n_\infty](t;p,q)$   of~\eqref{eq:master_inf-hom} as $M \to \infty$. Moreover, $[S^n](t;M)  - [S^n_{\infty}](t)  = O\left(\frac1M\right)$ as $M\to \infty$.
	\end{lemma}
	\noindent This lemma will be proved in Corollary~\ref{cor:Niu-hom} below. To simplify the notations, let
	\begin{subequations}
		\label{eq:u^M_n=[S^n](M)+infinity-hom}
		\begin{align}
			\label{eq:u^M_n=[S^n](M)-hom}
			&u^{(M)}_n(t):= [S^n](t;M), 
			\qquad 
			q^{(M)}_n:= \frac{M-n}{M-1}q, \qquad
			n=1, \dots, M,
			\\
			&    u^{\infty}_n(t):= [S^n_\infty](t), 
			\qquad n=1, 2, \dots,
		\end{align}
	\end{subequations}
	Then we can rewrite the systems~\eqref{eqs:masterfc} and~\eqref{eq:master_inf-hom} as 
	\begin{equation}
		\label{eq:odesys-M-hom}
		\frac{d}{dt}u^{(M)}_n(t)=-n\left(p+q^{(M)}_n\right)u^{(M)}_n+nq^{(M)}_nu^{(M)}_{n+1}.
		\qquad u^{(M)}_n(0)=1, \qquad 
		n=1, \dots, M,
	\end{equation}
	and 
	\begin{equation}
		\label{eq:limsys-hom}
		\frac{d}{dt}{u}_n^{\infty}(t)=-n(p+q) u_n^{\infty}+nq u_{n+1}^{\infty}, 
		\qquad u_n^{\infty}(0)=1,
		\qquad n=1,2, \dots,  
	\end{equation}
	respectively. 
		Note that ODE~\eqref{eq:odesys-M-hom} for~$\frac{d}{dt} u^{(M)}_M$ does not involve the non-existent variable $u^{(M)}_{M+1}$, since 
		\begin{equation}
			\label{eq:M+1term}
			q_M^{(M)}=0.
		\end{equation}
		
		For any fixed~$n$,
		\begin{equation}
			\label{eq:qnconv-hom}
			\lim_{M\to\infty} q^{(M)}_n=q.
		\end{equation}
		Hence, it is reasonable to expect that as $M\to\infty$, solutions of~\eqref{eq:odesys-M-hom} converge to solutions of the limiting infinite system~\eqref{eq:limsys-hom}.
		We cannot, however, prove this by applying the standard theorems for continuous dependence of ODE solutions on parameters, since the number of equations increases with~$M$ and becomes infinite in the limit, and because of the 
		presence of the unbounded factor~$n$ on the right-hand-sides 
		of~\eqref{eq:odesys-M-hom} and~\eqref{eq:limsys-hom}.

		It is convenient for the analysis to have an infinite number of
		ODEs even for a finite~$M$, because then $\{u_n^{(M)}\}$ and $\{u_n^{\infty}\}$ belong to the same space. Therefore, 
		we embed the finite system~\eqref{eq:odesys-M-hom} into the infinite system 
		\begin{subequations}
			\label{eq:odesys-M-embed-infinity-hom}
			\begin{equation}
				\label{eq:odesys-M-embed-infinity-ODE-hom}
				\frac{d}{dt}u^{(M)}_n(t)=-n\left(p+q^{(M)}_n\right)u^{(M)}_n+nq^{(M)}_nu^{(M)}_{n+1}, 
				\qquad u^{(M)}_n(0)=1, \qquad 
				n=1,2, \dots,
			\end{equation}
			where
			\begin{equation}
				\label{eq:odesys-q_n-M=infinity-hom}
				q^{(M)}_n=
				\begin{cases}
					\frac{M-n}{M-1}q, & \quad n=1, \dots, M, \\ 
					q, & \quad n=M+1, M+2,\dots 
				\end{cases}
			\end{equation}  
		\end{subequations}
		Condition~\eqref{eq:M+1term} ensures that the first $M$~components $\{u^{(M)}_n\}_{n=1}^M$ of the solution of~\eqref{eq:odesys-M-embed-infinity-hom} are also the solution  
		of~\eqref{eq:odesys-M-hom}. In addition, $u^{(M)}_n = u_n^{\infty}$ for $n=M+1, M+2, \dots$, since equations~\eqref{eq:limsys-hom} and~\eqref{eq:odesys-M-embed-infinity-hom} for these components
		are identical, and are decoupled from the first $M$~equations.
		Hence, solutions of the finite system~\eqref{eq:odesys-M-hom} converge to solutions of the
		limiting infinite system~\eqref{eq:limsys-hom} if and only if solutions of the infinite 
		system~\eqref{eq:odesys-M-embed-infinity-hom} converge to that limit. 
		
		The discussion so far takes for granted that solutions of the infinite systems~\eqref{eq:limsys-hom} and~\eqref{eq:odesys-M-embed-infinity-hom}
		exist. Even though those systems are linear, the existence of solutions is not quite trivial, because the
		presence of the factor~$n$ on the right-hand-sides 
		of~\eqref{eq:odesys-M-embed-infinity-ODE-hom} and~\eqref{eq:limsys-hom} makes those right-hand-sides
		unbounded functions of the infinite solution vectors
		\begin{equation}
			\label{eq:uform}
			\mathbf  u^{\infty}\mathrel{:=} (u_1^{\infty},u_2^{\infty},\cdots), \qquad 
			\mathbf  u^{(M)} \mathrel{:=} (u_1^{(M)},u_2^{(M)},\cdots).
		\end{equation}
		From the proof of Theorem~\ref{thm:Niu}, however, it follows that 
		\begin{align}
			u_n^{\infty} &= (1-f_{\rm Bass})^n, \qquad  n=1,2, \dots,
			\nonumber
			\\
			u_n^{(M)} &= (1-f_{\rm Bass})^n, \qquad n=M+1,M+2, \dots.
			\label{eq:u_n^(M)-for-n>M-hom}
		\end{align}
		Therefore, solutions of the infinite systems~\eqref{eq:limsys-hom} and~\eqref{eq:odesys-M-embed-infinity-hom} do
		exist.
		
		The following technical results will be used in the statement and proof of Theorem~\ref{thm:Steve-hom}.
		\begin{lemma}
			Let $p,q >0$, let $q^{(M)}_n$ 
			be given by~\eqref{eq:odesys-q_n-M=infinity-hom}, 
			and let
			\begin{equation}
				\label{eq:eps_tilde}
				\widetilde \epsilon\mathrel{:=} \ln(1+\tfrac pq).
			\end{equation}
			In addition, for every $\epsilon\ge 0$ and  ${\bf v} = (v_1, v_2, \dots) \in{\mathbb R}^\mathbb{N}$, define
			\begin{equation}
				\label{eq:thetadef}
				\theta(\epsilon)\mathrel{:=} \frac{q e^{\epsilon}}{p+q},
				\qquad 		
				\|\mathbf  v\|_{\epsilon}\mathrel{:=} \sup_{1\le n<\infty} e^{-\epsilon n}|v_n|, \qquad 
				\vertiii{\mathbf  v}_\epsilon \mathrel{:=} \sup_{t\ge0}\|\mathbf  v(t)\|_\epsilon.
			\end{equation}
			Then 
			\begin{equation}
				\label{eq:thetaest-hom}
				\theta(\epsilon)<1,\qquad 0\le\epsilon<\widetilde\epsilon.
			\end{equation}	
			In addition, the function $\| {\bf v} \|_\epsilon$ is a norm on the space 
			$
			{\bf V}_\epsilon  := \{{\bf v} \, \Huge| \, \| {\bf v}  \|_\epsilon <\infty  \},
			$
			and the
			function $\vertiii{{\bf v}(t)}_\epsilon$ is a norm on the space $C^0_B([0,\infty), {\bf V}_\epsilon)$ of bounded continuous functions on
			$[0,\infty)$  taking values in ${\bf V}_\epsilon$. 
			%
		\end{lemma}
		\begin{proof}
			These are standard results.
		\end{proof}

		\begin{theorem}
			\label{thm:Steve-hom}
			Let $p,q>0$ and $0<\epsilon<\widetilde\epsilon$. Let
			${\mathbf  u}^\infty$ and ${\mathbf  u}^{(M)}$  be the solutions of~\eqref{eq:limsys-hom} and~\eqref{eq:odesys-M-embed-infinity-hom}, respectively. Then 
			\begin{equation}
				\label{eq:limu-hom}
				%
				\vertiii{\mathbf  u^{(M)}-{\mathbf  u}^{\infty}}_\epsilon= O\left(\frac1M\right),
				\qquad M \to \infty.
			\end{equation}
	\end{theorem}
	
	\begin{proof}
		Since $\{u^{(M)}_n\}_{n=1}^M$ are probabilities, see~\eqref{eq:u^M_n=[S^n](M)-hom}, they are bounded between~0 
		and~1.
		In addition, $\{u^{(M)}_n\}_{n=M+1}^\infty$ are given by~\eqref{eq:u_n^(M)-for-n>M-hom},
		and so are also bounded between~0 and~1. Therefore, we have the uniform bound
		\begin{equation}
			\label{eq:ubnd-hom}
			\vertiii{{\mathbf  u}^{(M)}}_0 \le 1.
		\end{equation}
		Both here and in the other cases presented in this paper, the slighty weaker uniform bound
		$\vertiii{\myvec u^{(M)}}_0\le \frac{1}{1-\theta(0)}$ can be obtained directly from relevant system, in this case
		~\eqref{eq:odesys-M-embed-infinity-hom},  by
		multiplying by an integrating factor, integrating, applying the $\vertiii{\,}_0$ norm, and estimating the result, in similar fashion to
		the calculation below.

		Subtracting \eqref{eq:limsys-hom} from \eqref{eq:odesys-M-embed-infinity-hom} yields
		\begin{equation*}
			\begin{aligned}
				\frac{d}{dt} \left(u^{(M)}_n-u_n^{\infty}\right)&+n(p+q)\left(u^{(M)}_n-u_n^{\infty}\right)
				\\&\qquad=n q\left(u^{(M)}_{n+1}-u_{n+1}^{\infty}\right)+n\left(q-q^{(M)}_n\right)
				\left(u^{(M)}_n-u^{(M)}_{n+1}\right).
			\end{aligned}
		\end{equation*}
		Multiplying 
		by the integrating factor $e^{n(p+q)t}$, 
		integrating from zero to~$t$, and using
		$u^{(M)}_n(0)=u_n^{\infty}(0)$, 
		yields
		\begin{equation}
			\label{eq:diffinteq-hom}
			\begin{aligned}
				u^{(M)}_n(t)-u_n^{\infty}(t)&=n q \int_0^t e^{-n(p+q)(t-s)}\left(u^{(M)}_{n+1}(s)-u_{n+1}^{\infty}(s)\right)\,ds
				\\&\qquad+n \left(q-q^{(M)}_n\right)\int_0^t e^{-n(p+q)(t-s)}\left(u^{(M)}_n(s)-u^{(M)}_{n+1}(s)\right)\,ds.
			\end{aligned}
		\end{equation}
		Let $0<\epsilon<\widetilde \epsilon$. Multiplying both sides of the infinite system~\eqref{eq:diffinteq-hom} by~$e^{-n\epsilon}$, and taking the supremum over $n,t$ 
		gives
		\begin{equation*}
			\begin{aligned}
				&\vertiii{\mathbf  u^{(M)}-\mathbf  u^{\infty}}_\epsilon
				\le
				\sup_{t,n}e^{\epsilon} n q  \int_0^t e^{-n(p+q)(t-s)} \vertiii{\mathbf  u^{(M)}-\mathbf  u^\infty}_\epsilon\,ds
				\\&\qquad\qquad\qquad\qquad+\sup_{t,n} e^{-n\epsilon} n \left|q-q^{(M)}_n\right| \int_0^t e^{-n(p+q)(t-s)}  2\vertiii{\mathbf  u^{(M)}}_0 \,ds
				\\ 
				&\quad=\sup_{t,n} \frac{q e^{\epsilon}(1-e^{-n(p+q)t})}{p+q}\vertiii{\mathbf  u^{(M)}-\mathbf  u^\infty}_\epsilon
				\\&\qquad\qquad+2 \vertiii{\mathbf  u^{(M)}}_0 \sup_{t,n} \frac{ \left|q-q^{(M)}_n\right|e^{-n\epsilon}(1-e^{-n(p+q)t})}{p+q} 
				\\&\quad\le \theta(\epsilon) \vertiii{\mathbf  u^{(M)}-\mathbf  u^\infty}_\epsilon+
				2 \frac{\sup_n  \left(\left|q-q^{(M)}_n\right|e^{-n\epsilon}\right)}{p+q}, 
			\end{aligned}
		\end{equation*}
		where in the last inequality we used~\eqref{eq:ubnd-hom}.  
		Since $\theta(\epsilon)<1$, 
		see~\eqref{eq:thetaest-hom}, 
		then 
		\begin{equation}
			\label{eq:diffest2-hom}
			\vertiii{\mathbf  u^{(M)}-\mathbf  u^\infty}_\epsilon\le \frac{2 }{(p+q) 
				(1-\theta(\epsilon))} \sup_n \left(\left|q-q^{(M)}_n\right|e^{-n\epsilon}\right).
		\end{equation}
		By the definition of~$q^{(M)}_n$, see~\eqref{eq:odesys-q_n-M=infinity-hom}, for any $n \ge 1$, 
		$$
		\left|q-q^{(M)}_n\right|e^{-n\epsilon} \le q \frac{n-1}{M-1}e^{-n\epsilon}
		\le \frac{q}{\epsilon(M-1)} \left(n \epsilon e^{-n\epsilon}\right)
		\le \frac{q}{\epsilon(M-1)} C,
		$$
		where $C := \max_{y \ge 0}(y e^{-y})<\infty$.
		Therefore, 
		for any fixed $\epsilon>0$, 
		\begin{equation}
			\label{eq:q-q^m=O(1/M)-hom}
			\sup_n \left(\left|q-q^{(M)}_n\right|e^{-n\epsilon}\right)=O\left(\frac1M\right), \qquad M\to \infty.
		\end{equation}
		Relation~\eqref{eq:limu-hom} follows from~\eqref{eq:diffest2-hom} and~\eqref{eq:q-q^m=O(1/M)-hom}.
	\end{proof}

	\begin{corollary}
		\label{cor:Niu-hom}
		for any $n\in\NN$, $u^{(M)}_n(t)\to u_n^\infty(t)$ uniformly in $t$ as $M \to \infty$, and $u^{(M)}_n(t)- u_n^\infty(t)  = O\left(\frac1M\right)$ as $M \to \infty$. 
	\end{corollary}
	\begin{proof}
		For any $n\in\NN$ and $0 < \epsilon<\tilde{\epsilon}$.
		\begin{equation}
			\label{eq:|u_n|<|||u_n|||-hom}
			\left| u^{(M)}_n(t) -  u_n^\infty(t) \right| \le e^{n \epsilon}\vertiii{\mathbf  u^{(M)}-\mathbf  u^\infty}_\epsilon ,
		\end{equation}
		see~\eqref{eq:thetadef}. Therefore, the result follows from the convergence estimate \eqref{eq:limu-hom}.
	\end{proof}
	
		
			
			\section{Homogeneous circle}
			\label{sec:Bass-circle}
			To understand the {\em role of the network} in the discrete Bass model, it is instructive to consider the diffusion on the ``opposite of a complete network", namely, on a circle with $M$ nodes, where each node has only 2 edges. Thus, we assume that
			\begin{equation}
				\label{eq:circle_conditions}
				p_j\equiv p, \qquad d_j(M)\equiv 2,\qquad q_{i,j}=\begin{cases}
					{q},& i=j\pm 1 \mod M,\\
					0,& {\rm otherwise,}
				\end{cases},\qquad 1\leq i,j\leq M.
			\end{equation}
			In this case, \eqref{eq:discrete_model} reads
			\begin{equation}
				{\label{eq:general_circle}}
				{\rm Prob}(X_j(t+\Delta t)=1\ |\ {\bf X}(t))=\begin{cases}
					\hfill 1, \hfill & \qquad {\rm if}\ X_j(t)=1,\\
					\left(p+\frac{q}{2}X_{j-1}(t)+\frac{q}{2}X_{j+1}(t)\right)\Delta t, & \qquad {\rm if}\ X_j(t)=0.
				\end{cases}
			\end{equation}
			The $M\to\infty$ limit of the discrete Bass model on a circle was explicitly computed by Fibich and Gibori:
			\begin{theorem}[{\cite{OR-10}}]
				\label{thm:Fibich_Gibori}
				The expected fraction of adopters in the discrete Bass model~$(\ref{eq:general_initial},\ref{eq:general_circle})$ on a homogeneous circle as $M\to\infty$ is
				\begin{equation}
					\label{eq:f_Bass-Niu-circle}
					\lim\limits_{M\to\infty}f^{\rm circle}_{\rm discrete}(t;p,q,M)=f_{\rm 1D}(t;p,q),\qquad
					f_{\rm 1D}(t;p,q):=1-e^{-(p+q)t+q\frac{1-e^{-pt}}{p}}.
				\end{equation}
			\end{theorem}
			\begin{proof}
				We sketch the proof of~\cite{OR-10}. Let $[S^n_{\rm circle}](t)$ denote the probability that $n$ adjacent nodes $[S_{k+1},\dots,S_{k+n}]$ are all nonadopters at time~$t$. By translation invariance, $[S_{k+1},\dots,S_{k+n}]$ is independent of $k$. Hence, $$f^{\rm circle}_{\rm discrete}=1-[S^1_{\rm circle}].$$ By translation invariance and~\eqref{eq:circle_conditions}, the master equations~\eqref{eqs:masterHeterofc} reduce to
				\begin{subequations}
					\label{eq:mastereqs1d}
					\begin{align}
						\frac{d}{dt}[S^n_{\rm circle}]&=-(np+q)[S^n_{\rm circle}]+q[S^{n+1}_{\rm circle}],\qquad n=1,\dots,M-1,\\
						\frac{d}{dt}[S^M_{\rm circle}]&=-Mp[S^M_{\rm circle}],
					\end{align}
					subject to the initial conditions
					\begin{equation}
						[S^n_{\rm circle}](0)=1\qquad n=1,\dots,M.
					\end{equation}
				\end{subequations}
				Fixing $n$ and letting $M\to\infty$ in~\eqref{eq:mastereqs1d} gives the limiting system
				\begin{equation}
					\label{eq:mastereq1dinf}
					\frac{d}{dt}[S^n_{\infty,\rm circle}]=-(np+q)[S^n_{\infty,\rm circle}]+q[S^{n+1}_{\infty,\rm circle}],\qquad [S^n_{\infty,\rm circle}](0)=1,\qquad n=1,2,\dots
				\end{equation}
				The ansatz
				\begin{equation}
					\label{eq:substitution}
					[S^n_{\infty,\rm circle}]=e^{-(n-1)pt}[S_{\infty,\rm circle}]
				\end{equation}
				reduces~\eqref{eq:mastereq1dinf} into the single ODE
				\begin{equation*}
					\frac{d}{dt}[S_{\infty,\rm circle}]=-(p+q)[S_{\infty,\rm circle}]+qe^{-pt}[S_{\infty,\rm circle}], \qquad [S_{\infty,\rm circle}](0)=1.
				\end{equation*}
				Hence
				\begin{equation*}
					[S_{\infty,\rm circle}]= e^{-(p+q)t+q\frac{1-e^{-pt}}{p}}.
				\end{equation*}
				Since $\lim\limits_{M\to\infty} f^{\rm circle}_{\rm discrete}(t;p,q,M)=1-[S_{\infty,\rm circle}]$, the result follows.
				
				Note, however, that in~\cite{OR-10}, Fibich and Gibori did not rigorously justify that~\eqref{eq:mastereq1dinf} is the limit of~\eqref{eq:mastereqs1d} as $M\to\infty$. We will justify this limit, and thus complete the proof of Theorem~\ref{thm:Fibich_Gibori}, in Section~\ref{subsec:conv_rate_circle}.
			\end{proof}
			
			\subsection{Convergence and rate of convergence}
			\label{subsec:conv_rate_circle}
			\begin{theorem}
				\label{thm:circle_conv}
				Assume the conditions of Theorem~\ref{thm:Fibich_Gibori}. Then the limit~\eqref{eq:f_Bass-Niu-circle} is uniform in~$t$. Moreover, the rate of convergence is exponential, i.e., for any $0<\epsilon<\widetilde\epsilon$, where $\widetilde\epsilon$ is given by~\eqref{eq:eps_tilde},
				\begin{equation}
					\label{eq:f_Bass-rate-of-convergence-hom-circle}
					f^{\rm circle}_{\rm discrete}(t;p,q,M) - f_{\rm 1D}(t;p,q) = O\left(e^{-M\epsilon}\right), \qquad M \to \infty.
				\end{equation}
			\end{theorem}
			\begin{proof}
				In light of the proof of Theorem~\ref{thm:Fibich_Gibori}, it only remains to show that for any $n \in \mathbb{N}$ and $p,q>0$:
				\begin{enumerate}
					\item
					The solution $[S^n_{\rm circle}](t;p,q,M)$  of~\eqref{eq:mastereqs1d}
					converges, uniformly in~$t$, to the solution $[S^n_{\infty,\rm circle}](t;p,q)$ of~\eqref{eq:mastereq1dinf} as $M \to \infty$.
					\item
					The rate of convergence is exponential, i.e.,
					\begin{equation*}
						[S^n_{\rm circle}](t;M)  - [S^n_{\infty,\rm circle}](t)  = O\left(e^{-M\epsilon}\right), \qquad M \to \infty.
					\end{equation*}
				\end{enumerate} 
				These results are proved in Corollary~\ref{cor:Niu-circle} below.
			\end{proof}
			The exponential rate of convergence predicted in Theorem~\ref{thm:circle_conv} is illustrated numerically in Figure~\ref{fig:convergence_circle}, where we observe that $f_{\rm 1D}-f^{\rm circle}_{\rm discrete}\approx Ce^{-1.3M}$. Here $f^{\rm circle}_{\rm discrete}(t;p,q,M)$ was calculated using the explicit expression obtained in~\cite{OR-10,Bass-boundary-18}.
			\begin{figure}[ht!]
				\begin{center}
					\scalebox{1}{\includegraphics{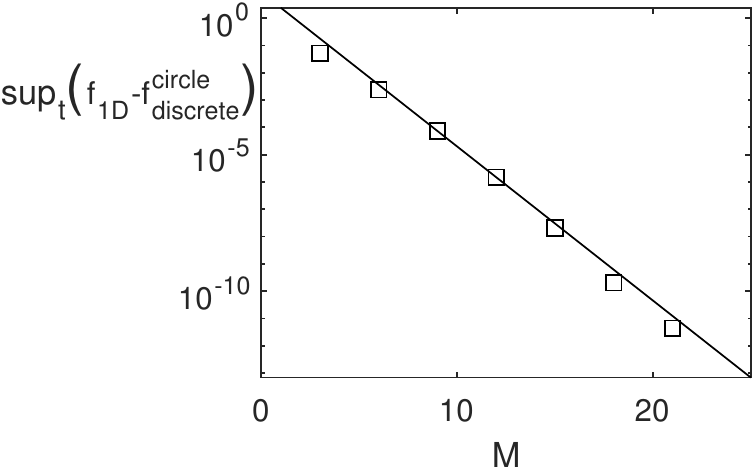}}
					\caption{Semi-log plot of $\sup_{t}\left(f_{\rm 1D}-f^{\rm circle}_{\rm discrete}\right)$ as a function of $M$. The fitted solid line is $y=8.8e^{-1.3M}$. Here, $p=0.02$ and $q=0.11$.}
					\label{fig:convergence_circle}
				\end{center}
			\end{figure}
			
			The analysis that leads to Corollary~\ref{cor:Niu-circle} is nearly identical to that in Section~\ref{sec:proof-Niu}, thus illustrating the power of our method. Let
			
			\begin{subequations}
				
				\label{eq:u^M_n=[S^n](M)+infinity}	
				\begin{align}		
					\label{eq:u^M_n=[S^n](M)}	
					&u^{(M)}_n(t):= [S^n_{\rm circle}](t;M),		
					\qquad		
					q^{(M)}_n:= \begin{cases}
						q, & n=1,\dots,M-1,\\
						0, & n=M,
					\end{cases} \qquad		
					n=1, \dots, M,		
					\\	
					&    u^{\infty}_n(t):= [S^n_{\infty,\rm circle}](t),	
					\qquad n=1, 2, \dots.	
				\end{align}	
			\end{subequations}
			Then we can rewrite the systems~\eqref{eq:mastereqs1d} and~\eqref{eq:mastereq1dinf} as
			
			\begin{equation}	
				\label{eq:odesys-M}
				\frac{d}{dt}u^{(M)}_n(t)=-\left(np+q^{(M)}_n\right)u^{(M)}_n+q^{(M)}_nu^{(M)}_{n+1}.
				\qquad u^{(M)}_n(0)=1, \qquad
				n=1, \dots, M,
			\end{equation}
			and
			\begin{equation}
				\label{eq:limsys}
				\frac{d}{dt}{u}_n^{\infty}(t)=-(np+q) u_n^{\infty}+q u_{n+1}^{\infty},
				\qquad u_n^{\infty}(0)=1,
				\qquad n=1,2, \dots, 
			\end{equation}
			respectively.
			We embed~\eqref{eq:odesys-M} into the infinite system
			\begin{subequations}	
				\label{eq:odesys-M-embed-infinity-circle}	
				\begin{equation}			
					\label{eq:odesys-M-embed-infinity-ODE-circle}			
					\frac{d}{dt}u^{(M)}_n(t)=-\left(np+q^{(M)}_n\right)u^{(M)}_n+q^{(M)}_nu^{(M)}_{n+1},			
					\qquad u^{(M)}_n(0)=1, \qquad			
					n=1,2, \dots,			
				\end{equation}	
				where	
				\begin{equation}			
					\label{eq:odesys-q_n-M=infinity}			
					q^{(M)}_n=			
					\begin{cases}				
						q, & \quad n\neq M, \\				
						0, & \quad n=M.		
					\end{cases}			
				\end{equation} 	
			\end{subequations}
			By~\eqref{eq:substitution},
			\begin{align}		
				u_n^{\infty} &= e^{-(n-1)pt}[S_{\infty,\rm circle}], \qquad  n=1,2, \dots,		
				\nonumber		
				\\		
				u_n^{(M)} &= e^{-(n-1)pt}[S_{\infty,\rm circle}], \qquad n=M+1,M+2, \dots.		
				\label{eq:u_n^(M)-for-n>M}		
			\end{align}
			
			Therefore, solutions of the infinite systems~\eqref{eq:limsys} and~\eqref{eq:odesys-M-embed-infinity-circle} do exist.
			\begin{lemma}
				\label{thm:Steve}
				Let $p,q>0$ and $0<\epsilon<\widetilde\epsilon$, where $\widetilde\epsilon$ is given by~\eqref{eq:eps_tilde}. Let ${\mathbf  u}^\infty$ and ${\mathbf  u}^{(M)}$  be the solutions of~\eqref{eq:limsys} and~\eqref{eq:odesys-M-embed-infinity-circle}, respectively. Then
				\begin{equation}
					\label{eq:limu-circle}
					%
					\vertiii{\mathbf  u^{(M)}-{\mathbf  u}^{\infty}}_\epsilon= O\left(e^{-M\epsilon}\right),			
					\qquad M \to \infty.			
				\end{equation}
				
			
		\end{lemma}
		
		\begin{proof}
			Since $\{u^{(M)}_n\}_{n=1}^M$ are probabilities, see~\eqref{eq:u^M_n=[S^n](M)}, they are bounded between~0 and~1. In addition, $\{u^{(M)}_n\}_{n=M+1}^\infty$ are given by~\eqref{eq:u_n^(M)-for-n>M}, and so are also bounded between~0 and~1. Therefore, we have the uniform bound
			\begin{equation}		
				\label{eq:ubnd-circle}		
				\vertiii{{\mathbf  u}^{(M)}}_0 \le 1.		
			\end{equation}
			
			Subtracting \eqref{eq:limsys} from \eqref{eq:odesys-M} yields
			\begin{equation*}		
				\begin{aligned}			
					\frac{d}{dt} \left(u^{(M)}_n-u_n^{\infty}\right)&+(np+q)\left(u^{(M)}_n-u_n^{\infty}\right)			
					\\&\qquad=q\left(u^{(M)}_{n+1}-u_{n+1}^{\infty}\right)+\left(q-q^{(M)}_n\right)			
					\left(u^{(M)}_n-u^{(M)}_{n+1}\right).		
				\end{aligned}		
			\end{equation*}
			Multiplying 
			by the integrating factor $e^{(np+q)t}$, integrating from zero to~$t$, and using $u^{(M)}_n(0)=u_n^{\infty}(0)$, yields
			
			\begin{align}		
				\label{eq:diffinteq-circle}				
				u^{(M)}_n(t)-u_n^{\infty}(t)&=q \int_0^t e^{-(np+q)(t-s)}\left(u^{(M)}_{n+1}(s)-u_{n+1}^{\infty}(s)\right)\,ds			
				\\&\qquad+ \left(q-q^{(M)}_n\right)\int_0^t e^{-(np+q)(t-s)}\left(u^{(M)}_n(s)-u^{(M)}_{n+1}(s)\right)\,ds.	\nonumber		
			\end{align}
			Let $0<\epsilon<\widetilde \epsilon$. Multiplying both sides of the infinite system~\eqref{eq:diffinteq-circle} by~$e^{-n\epsilon}$, and taking the supremum over $n,t$
			gives
			
			\begin{equation*}		
				\begin{aligned}			
					&\vertiii{\mathbf  u^{(M)}-\mathbf  u^{\infty}}_\epsilon			
					\le			
					\sup_{t,n}e^{\epsilon} q  \int_0^t e^{-(np+q)(t-s)} \vertiii{\mathbf  u^{(M)}-\mathbf  u^\infty}_\epsilon\,ds			
					\\&\qquad\qquad\qquad\qquad+\sup_{t,n} e^{-n\epsilon} \left|q-q^{(M)}_n\right| \int_0^t e^{-(np+q)(t-s)}  2\vertiii{\mathbf  u^{(M)}}_0 \,ds			
					\\			
					&\quad=\sup_{t,n} \frac{q e^{\epsilon}(1-e^{-(np+q)t})}{np+q}\vertiii{\mathbf  u^{(M)}-\mathbf  u^\infty}_\epsilon			
					\\&\qquad\qquad+2 \vertiii{\mathbf  u^{(M)}}_0 \sup_{t,n} \frac{ \left|q-q^{(M)}_n\right|e^{-n\epsilon}(1-e^{-(np+q)t})}{np+q}			
					\\&\quad\le \theta(\epsilon) \vertiii{\mathbf  u^{(M)}-\mathbf  u^\infty}_\epsilon+		
					2 \frac{\sup_n  \left(\left|q-q^{(M)}_n\right|e^{-n\epsilon}\right)}{p+q},		
				\end{aligned}	
			\end{equation*}
			where in the last inequality we used~\eqref{eq:ubnd-circle}. Since $\theta(\epsilon)<1$, 
			see~\eqref{eq:thetaest-hom}, then 
			
			\begin{equation}		
				\label{eq:diffest2-circle}		
				\vertiii{\mathbf  u^{(M)}-\mathbf  u^\infty}_\epsilon\le \frac{2 }{(p+q) 
					(1-\theta(\epsilon))} \sup_n \left(\left|q-q^{(M)}_n\right|e^{-n\epsilon}\right).		
			\end{equation}
			By the definition of~$q^{(M)}_n$, see~\eqref{eq:odesys-q_n-M=infinity},
			\begin{equation}	
				\label{eq:q-q^m=O(e-M)}	
				\sup_n \left(\left|q-q^{(M)}_n\right|e^{-n\epsilon}\right)=\left|q-q^{(M)}_M\right|e^{-M\epsilon}=qe^{-M\epsilon}.
			\end{equation}
			Relation~\eqref{eq:limu-circle} follows from~\eqref{eq:diffest2-circle} and~\eqref{eq:q-q^m=O(e-M)}.
		\end{proof}
		
		\begin{corollary}
			\label{cor:Niu-circle}
			For any $n\in\NN$ and $0<\epsilon<\tilde{\epsilon}$, $u^{(M)}_n(t)\to u_n^\infty(t)$ uniformly in $t$ as $M\to\infty$, and  $u^{(M)}_n(t)- u_n^\infty(t)  = O\left(e^{-M\epsilon}\right)$ as $M \to \infty$.
		\end{corollary}
		\begin{proof}
			Same as Corollary~\ref{cor:Niu-hom}, only with the convergence estimate~\eqref{eq:limu-circle}.
		\end{proof}

		\section{Heterogeneous network with $K$ homogeneous groups}
		\label{sec:compartmental}
		
		We now consider a heterogeneous population that consists of $K$ groups, each of which is homogeneous. 
		This situation arises, e.g., when we divide the population according to age groups,
		levels of income, gender, etc. 
		
		\subsection{Description of the model}
		
		The network consists of $K$ groups. Any node in group $k$ has external influence $p_k> 0$, any adopter in group $m$ has internal influence $\frac{q_{m,k}}{M-1}\geq 0$ on any nonadopter in group $k$, and the parameters $\{p_j\}$ and $\{q_{m,j}\}$ 
		are assumed to satisfy
		\begin{equation}
			\label{eq:pqge0}
			p_j>0, \quad q_{m,j}\ge0 \qquad\text{for $m,j\in\{1,\cdots,K\}$.}
		\end{equation}
		Group $k$ has $M_k>0$ nodes, and so $\sum_{k=1}^{K}M_k=M$. 
		
		The discrete model~\eqref{eq:discrete_model} for node $k_n$ in group $k$, where $k=1,\dots,K$ and $n=1,\dots,M_k$, becomes
		\begin{equation}
			\label{eq:discrete_group_model}
			{\rm Prob}(X_{k_n}(t+\Delta t)=1\ |\ {\bf X}(t))=\begin{cases}
				\hfill 1,\hfill &\qquad {\rm if}\ X_{k_n}(t)=1,\\
				\left(p_{k}+\sum\limits_{m=1}^K\frac{q_{m,k}}{M-1}N_{m}(t)\right)\Delta t, &\qquad {\rm if}\ X_{k_n}(t)=0,
			\end{cases}
		\end{equation}
		where $N_m(t)$ is the number of adopters in group $m$.
		The corresponding {\em heterogeneous compartmental Bass model} reads, cf.~\eqref{eq:Bass_intro},
		\begin{equation}
			\label{eq:heterogeneous_Bass_n}
			\frac{d}{dt}n_k(t)=(M_k-n_k)\left(p_k+\sum_{m=1}^{K}\frac{q_{m,k}}{M}n_m\right), \qquad n_k(0)=0, \qquad 1\leq k\leq K,
		\end{equation}
		where $n_k$ is the number of adopters in group $k$.
		Although complete consistency with \eqref{eq:discrete_group_model} would require dividing $q_{m,k}$ in \eqref{eq:heterogeneous_Bass_n} by $M-1$ rather than $M$, the use of $M$ as the
		denominator is traditional in Bass models, going back to Bass~\cite{Bass-69}. In any case, the difference becomes $O\left(\frac1M\right)$ as $M\to\infty$.
		
		Let us assume that the following limits exist and are positive: 
		\begin{equation}
			\label{eq:a_k_def}
			a_k:=\lim\limits_{M\to\infty}\frac{M_k}{M}, \qquad k=1,\ldots,K.
		\end{equation}
		Then $a_k>0$ and $\sum_{k=1}^{K}a_k=1$. Moreover, we assume that
		\begin{equation}
			\label{eq:pop_frac_lim}
			\frac{M_k}{M}-a_k=O\left(\frac1{M}\right) \qquad \text{as } M\to\infty, \qquad 1\leq k\leq K.
		\end{equation}
		Relation~\eqref{eq:pop_frac_lim} holds trivially when $K=1$. For $K\geq 2$, it holds if, e.g., $M_k=\lfloor a_k M\rfloor$ for $1\le k\le K-1$ and
		$M_K=M-\sum_{k=1}^{K-1}M_k$.
		
		Let $f_k:=\frac{n_k}{M}$ denote the fraction within the population of group-$k$ adopters. Then \eqref{eq:heterogeneous_Bass_n}
		can be rewritten as
		\begin{equation}
			\label{eq:heterogeneous_Bass_M}
			\frac{d}{dt}f_k(t)=(\tfrac{M_k}{M}-f_k)\left(p_k+\sum_{m=1}^{K}q_{m,k}f_m\right), \qquad f_k(0)=0, \qquad 1\leq k\leq K.
		\end{equation}
		Since the number of equations in~\eqref{eq:heterogeneous_Bass_M} remains fixed as $M\to\infty$, assumption~\eqref{eq:a_k_def}, together with standard results on the continuous dependence of solutions of systems of ODEs
		on parameters and the fact that $f_k$ tends to the constant value~$\frac{M_k}M$ as $t\to\infty$,
		imply that the solution of~\eqref{eq:heterogeneous_Bass_M} tends uniformly in time as $M$ tends to infinity to the
		solution of 
		\begin{equation}
			\label{eq:heterogeneous_Bass}
			\frac{d}{dt}f_k(t)=(a_k-f_k)\left(p_k+\sum_{m=1}^{K}q_{m,k}f_m\right), \qquad f_k(0)=0, \qquad 1\leq k\leq K,
		\end{equation}
		and assumption~\eqref{eq:pop_frac_lim} implies that the rate of convergence is $O(\frac{1}{M})$.
		
		When $K=1$, both~\eqref{eq:heterogeneous_Bass_M} and~\eqref{eq:heterogeneous_Bass} reduce to~\eqref{eq:homogeneous_Bass}.
		Since the solution $f_k$ of \eqref{eq:heterogeneous_Bass} satisfies $0\le f_k\le a_k$, the maximal influence on a nonadopter from group $k$ is, c.f.~\eqref{eq:q_on_node}, 
		\begin{equation}
			\label{eq:max_q_compartmental}
			q_k:=\sum_{m=1}^{K}a_mq_{m,k}.
		\end{equation} 
		
		\subsection{Convergence and rate of convergence}
		\label{subsec:compartmental_convergence}
		We can use our method to prove the convergence and the rate of convergence of the $K$-groups heterogeneous discrete Bass model to the heterogeneous compartmental Bass model.
		
		\begin{theorem}
			\label{thm:fc_het_M->infinity}
			Let~\eqref{eq:pqge0} and~\eqref{eq:pop_frac_lim} hold. Then the expected fraction of adopters in the discrete Bass model~$(\ref{eq:general_initial},\ref{eq:discrete_group_model})$ on a heterogeneous network with $K$ groups approaches that of the corresponding heterogeneous compartmental Bass model~\eqref{eq:heterogeneous_Bass} as $M\to\infty$, i.e.,
			\begin{equation}
				\label{eq:disc_to_compartmental_lim}
				\lim\limits_{M\to\infty}f^{\rm heter}_{\rm discrete}(t;\{p_k\},\{q_{i,k}\},\{a_k\},M)=f^{\rm heter}_{\rm Bass}(t;\{p_k\},\{q_{i,k}\},\{a_k\}),
			\end{equation}
			where $f^{\rm heter}_{\rm Bass}:=\sum_{k=1}^{K}f_k$,
			$\{f_k\}$ are the solutions of~\eqref{eq:heterogeneous_Bass}, and the limit is uniform in~$t$. Moreover, the {\em rate of convergence} is $\frac1{M}$, i.e.,
			\begin{equation}
				\label{eq:f_Bass-rate-of-convergence}
				f^{\rm heter}_{\rm Bass}(t)-f^{\rm heter}_{\rm discrete}(t)= O\left(\frac1{M}\right), \qquad M \to \infty.
			\end{equation}
		\end{theorem}
		\begin{proof}
			Let $\vec{M}:=\left(M_1,\ldots,M_K\right)^T$, $\vec k:=(k_1,\cdots, k_K)^T$,
			\begin{equation}
				\label{eq:n_def}
				n(\vec k):=\sum_{j=1}^{K}k_j,	
			\end{equation}
			and
			\begin{equation*}
				\vec K^{\vec M}:=\{\vec k\ |\ k_j\in\{0,\dots,M_j\}, \quad 1\leq n(\vec k)\leq M\}.
			\end{equation*}
			
			Let $u_{\vec k}(t)$ denote the probability that a set $\{m_1,\dots,m_n\}$ of $n=n(\vec k)$ consumers which contains $k_j$ members from group $j$ for $1\leq j\leq K$, are all non-adopters at time~$t$. The master equations~\eqref{eqs:masterHeterofc} reduce to
			\begin{subequations}
				\label{eqs:pode}
				\begin{equation}
					\label{eq:pode}
					\tddt u_{\vec k}(t)=-\left(\vec k\cdot \vec p + \frac{\left(\vec{M}-\vec{k}\right)}{M-1}^TQ{\vec k}\right) u_{\vec k}
					+\frac{\left(\vec{M}-\vec{k}\right)}{M-1}^TD\left(u_{\vec k+\vec e_1},\cdots,u_{\vec k+\vec e_K}\right)Q{\vec k},
				\end{equation}
				for $1\leq n(\vec k)\leq M-1$, and
				\begin{equation}
					\label{eq:podeM}
					\tddt u_{\vec M}(t)=-\left(\vec{M}\cdot \vec{p}\right) u_{\vec M},
				\end{equation}
				subject to
				\begin{equation}
					\label{eq:init}
					u_{\vec k}(0)=1, \qquad \vec k\in \vec K^{\vec M},
				\end{equation}
				where $\vec e_i$ is the unit vector in the $i$th coordinate,
				\begin{gather*}
					\label{eq:Q}
					Q=\left(q_{i,j}\right)=\begin{pmatrix}q_{1,1}&\cdots& q_{1,K}\\\vdots &\ddots&\vdots\\ q_{K,1}&\cdots&q_{K,K}\end{pmatrix},
					\qquad
					D\left(\lambda_1,\cdots, \lambda_n\right)=
					\begin{pmatrix}
						\lambda_1&&\text{\large $0$}\\ &\ddots\\\text{\large $0$}&&  \lambda_n
					\end{pmatrix}.
				\end{gather*}
			\end{subequations}
			
			If we formally fix~$\vec k$, let $M\to\infty$ in~\eqref{eqs:pode}, and use~\eqref{eq:a_k_def}, we get that
			\begin{equation}
				\label{eq:master_inf}
				\tddt u_{\vec k}^{(\infty)}=-\left(\vec k\cdot \vec p + \vec a^T Q{\vec k}\right) u_{\vec k}^{(\infty)}
				+ \vec a^T   D\left(u_{\vec k+\vec e_1}^{(\infty)},\cdots,u_{\vec k+\vec e_K}^{(\infty)}\right)Q{\vec k}, \qquad u_{\vec k}^{(\infty)}(0)=1,
			\end{equation}
			where $\vec a:=\left(a_1,\ldots,a_K\right)^T$. This does not immediately imply that
			\begin{equation}
				\label{eq:S^M->S-complete}
				\lim_{M \to \infty} u_{\vec k} =  u_{\vec k}^{(\infty)},
			\end{equation}
			since the number of ODEs in~\eqref{eqs:pode} increases with~$M$, and becomes infinite in the limit. In Lemma~\ref{lem:Steve-convergence-het} below, however,
			we will prove that the limit~\eqref{eq:S^M->S-complete} holds for any $\vec k\in\vec K^{\infty}$, where
			\begin{equation*}
				\vec K^{\infty}:=\{\vec k\ |\ k_j\in\{0,1,2,\dots\}, \quad 1\leq n(\vec k)<\infty\}.
			\end{equation*}
			Therefore, we can proceed to solve the infinite system~\eqref{eq:master_inf}. 
			
			Let
			\begin{subequations}
				\label{eqs:subbed_equation}
				\begin{equation}
					\label{eq:inf_sol_fc}
					u_{\vec k}^{(\infty)}(t):=\prod_{j=1}^K \left({u_{\vec e_j}^{(\infty)}}(t)\right)^{k_j},
				\end{equation}
				where
				\begin{equation}
					\label{eq:subbed_equation}
					\tddt u_{\vec e_j}^{(\infty)}(t)=-\left(p_j + \sum_{i=1}^{K}a_iq_{i,j}\left(1-u_{\vec e_i}^{(\infty)}\right)\right) u_{\vec e_j}^{(\infty)}, \qquad u_{\vec e_j}^{(\infty)}(0)=1, \quad j=1,\dots,K.
				\end{equation}
			\end{subequations}	
			It can be verified by direct substitution that~\eqref{eqs:subbed_equation} satisfies~\eqref{eq:master_inf} for any $\vec k\in\vec K^{\infty}$. 
			
			Let
			\begin{equation*}
				\label{eq:system_equivalence}
				f_j:=a_j\left(1-u_{\vec e_j}^{(\infty)}\right), \qquad j=1,\dots,K.
			\end{equation*}
			By~\eqref{eq:subbed_equation}, $\{f_j\}_{j=1}^K$ satisfy the heterogeneous compartmental Bass model~\eqref{eq:heterogeneous_Bass}. Hence,
			\begin{subequations}
				\label{eqs:substitutions}
				\begin{equation}
					\label{eq:sub_1}
					\sum_{k=1}^{K}a_ku_{\vec e_k}^{(\infty)}=\sum_{k=1}^{K}\left(a_k-f_k\right)=1-f^{\rm heter}_{\rm Bass}.
				\end{equation}
				In addition, by~\eqref{eq:expected_frac},
				\begin{equation}
					\label{eq:sub_2}
					\sum_{k=1}^{K}\frac{M_k}{M}u_{\vec e_k}=\frac{1}{M}\sum_{k=1}^{M}[S_k]=1-f^{\rm heter}_{\rm discrete}.
				\end{equation}
			\end{subequations}
			Therefore,~\eqref{eq:disc_to_compartmental_lim} and~\eqref{eq:f_Bass-rate-of-convergence} follow from~\eqref{eqs:substitutions},~\eqref{eq:pop_frac_lim}, and Lemma~\ref{lem:Steve-convergence-het}.
		\end{proof}
		The $O\left(\frac1M\right)$ rate of convergence predicted in Theorem~\ref{thm:fc_het_M->infinity} is illustrated numerically in Figure~\ref{fig:convergence}, where we observe that $f^{\rm heter}_{\rm Bass}-f^{\rm heter}_{\rm discrete}\approx\frac{C}{M^{0.96}}$. Here $f^{\rm heter}_{\rm discrete}(t;\{p_k\},\{q_{m,k}\},M)$ was calculated from $10^6$ simulations of~$(\ref{eq:general_initial},\ref{eq:discrete_group_model})$.
		\begin{figure}[ht!]
			\begin{center}
				\scalebox{1}{\includegraphics{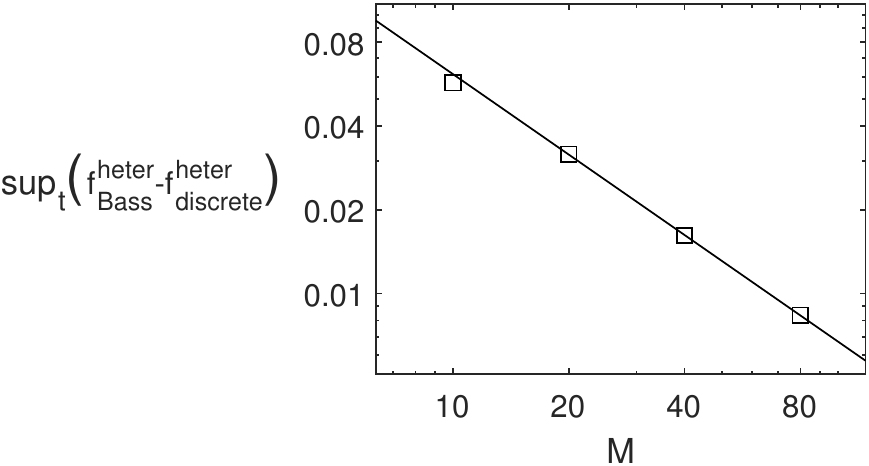}}
				\caption{Log-log plot of $\sup_{t}\left(f^{\rm heter}_{\rm Bass}-f^{\rm heter}_{\rm discrete}\right)$ as a function of $M$. The fitted solid line is $\log y=-0.96\log M-0.58$. Here, $K=4$, $\{a_1, a_2, a_3, a_4\}=\{0.4, 0.1, 0.3, 0.2\}$, $\{p_1, p_2, p_3, p_4\}=\{0, 0.02, 0.04, 0.01\}$, and \usebox{\smlmat}.}
				\label{fig:convergence}
			\end{center}
		\end{figure}
		\subsection{Proof of Lemma~\ref{lem:Steve-convergence-het}} 
		\label{sec:convergence}
		As noted, the proof of Theorem~\ref{thm:fc_het_M->infinity} makes use of the following result:
		\begin{lemma}
			\label{lem:Steve-convergence-het}
			For any $\vec k\in\vec K^{\infty}$, the solution $u_{\vec k}(t)$  of~\eqref{eqs:pode}
			converges, uniformly in~$t$, to the solution $u_{\vec k}^{(\infty)}(t)$   of~\eqref{eq:master_inf} as $M \to \infty$. Moreover, 
			\begin{equation}
				\label{eq:u_rate_convergence}
				u_{\vec k}-u_{\vec k}^{(\infty)}=O\left(\frac1{M}\right), \qquad M\to\infty.
			\end{equation}
			
		\end{lemma}
		\noindent This result is proved in Corollary~\ref{cor:Niu} below.
		
		\begin{remark}
			In some cases, it makes sense to replace~\eqref{eq:pop_frac_lim} with a more general convergence rate 
			\begin{equation}
				\label{eq:pop_frac_lim-r}
				\frac{M_k}{M}-a_k=O\left(\frac1{M^{r}}\right) \qquad \text{as } M\to\infty, \qquad 1\leq k\leq K.
			\end{equation}
			For example, if each person is assigned randomly to group $k$ with
			probability~$a_k$ then large deviations theory together with the Borel-Cantelli lemma imply that estimate
			\eqref{eq:pop_frac_lim-r} holds almost surely for any $r<\frac12$.
			In that case, the rate of convergence~\eqref{eq:f_Bass-rate-of-convergence} and~\eqref{eq:u_rate_convergence} changes from 
			$1/M$ to $1/M^{\min(1,r)}$.
		\end{remark}
		We now rigorously show that for any $\vec k \in \vec K^{\infty}$, the solution of~\eqref{eqs:pode} approaches the solution of~\eqref{eq:master_inf} as $M \to \infty$. Similarly to the homogeneous case~(Section~\ref{sec:proof-Niu}), we first embed the finite system~\eqref{eqs:pode} into the infinite system 
		\begin{subequations}
			\label{eq:odesys-M-embed-infinity}
			\begin{equation}
				\label{eq:odesys-M-embed-infinity-ODE}
				\begin{aligned}
					\tddt u_{\vec k}^{(\vec M)}=-\left(\vec k\cdot \vec p + \vec a(\vec M,\vec k)Q\vec k\right) u_{\vec k}^{(\vec M)}
					+ \vec a(\vec M,\vec k)D&\left(u_{\vec k+\vec e_1}^{(\vec M)},\cdots,u_{\vec k+\vec e_K}^{(\vec M)}\right)Q\vec k,\\& 
					\quad
					u_{\vec k}^{(\vec M)}(0)=1, \qquad 
					\vec k\in \vec K^{\infty},
				\end{aligned}
			\end{equation}
			where
			\begin{equation}
				\label{eq:addalp}
				\vec a(\vec M,\vec k)\eqdef
				\begin{cases}
					\frac{\left(\vec{M}-\vec{k}\right)}{M-1}^T,&\qquad \vec k\in \vec K^{\vec M},\\
					\vec a^T, &\qquad \vec k\in \vec K^{\infty}\backslash \vec K^{\vec M}.
				\end{cases}	
			\end{equation} 
		\end{subequations}
		Thus, for $\vec k\in \vec K^{\vec M}$, the ODEs of the infinite systems~\eqref{eq:odesys-M-embed-infinity} and~\eqref{eqs:pode} are identical. The ODEs of~\eqref{eq:odesys-M-embed-infinity} for $\vec k\in \vec K^{\vec M}$ are decoupled from those for $\vec k\in \vec K^{\infty}\backslash \vec K^{\vec M}$, since $\vec a(\vec M,\vec k)(j)=0$ when $k_j=M_j$ for any $1\leq j\leq K$. In addition, the ODEs of the infinite systems~\eqref{eq:odesys-M-embed-infinity} and~\eqref{eq:master_inf} for $\vec k\in \vec K^{\infty}\backslash \vec K^{\vec M}$ are identical, and are decoupled from the corresponding ODEs for $\vec K^{\vec M}$. Therefore $u_{\vec k}^{(\vec M)}\equiv u_{\vec k}^{(\infty)}$ for $\vec k\in \vec K^{\infty}\backslash \vec K^{\vec M}$. Hence, solutions of the finite~system~\eqref{eqs:pode} converge to solutions of the limiting system~\eqref{eq:master_inf} if and only if solutions of the infinite system~\eqref{eq:odesys-M-embed-infinity} converge to that limit.

		From the proof of Theorem~\ref{thm:fc_het_M->infinity}, it follows that 
		\begin{align}
			u_{\vec k}^{(\infty)} &= \prod_{j=1}^{K} (1-\frac{f_j}{a_j})^{k_j}, \qquad  \vec k\in\vec K^{\infty},
			\nonumber
			\\
			u_{\vec k}^{(\vec M)} &= \prod_{j=1}^{K} (1-\frac{f_j}{a_j})^{k_j}, \qquad \vec k\in \vec K^{\infty}\backslash \vec K^{\vec M}.
			\label{eq:u_n^(M)-for-n>M-het}
		\end{align}
		Therefore, solutions of the infinite systems~\eqref{eq:master_inf} and~\eqref{eq:odesys-M-embed-infinity} do
		exist.
		
		The following technical results will be used in the statement and proof of Theorem~\ref{thm:convergence}.
		
		\begin{lemma}\label{lem:elem}
			Let $\{p_j\}$ and $\{q_{m,j}\}$ satisfy \eqref{eq:pqge0}, let $\vec a(\vec M,\vec k)$ be given by~\eqref{eq:addalp}, and let
			\begin{equation}
				\label{eq:qnbnd}
				A\eqdef \sup_{\vec M\in\NN^K, \vec k\in\vec K^{\infty},} \max_{\omega_j\ge0, \sum_{j=1}^{K}\omega_j=1}\vec a(\vec M,\vec k)Q\vec \omega,
			\end{equation}
			where $\vec{\omega}=\left(\omega_1,\dots,\omega_K\right)^T$ and $Q$ is defined in \eqref{eq:Q}.
			Then $0<A<\infty$. 
			
			Define
			\begin{equation}\label{eq:supeps}
				\widetilde \eps\eqdef \ln(1+\tfrac{\min_j p_j}A), \qquad
				\theta(\eps)\eqdef \frac{A e^{\eps}}{\min_j p_j+A}.
			\end{equation}
			Also, for every $\eps\ge0$ and $\myvec v=\left(v_1,v_2,\dots\right)\in K^{\NN}$, define
			\begin{equation}\label{eq:epsnorm}
				\|\myvec v\|_{\eps}\eqdef \sup_{\vec k\in\vec K^{\infty}} e^{-\eps n(\vec k)}|v_{\vec k}|, \qquad 
				\vertiii{\myvec v}_\eps \eqdef \sup_{t\ge0}\|\myvec v(t)\|,
			\end{equation}
			where $n(\vec k)$ is defined in~\eqref{eq:n_def}.
			Then for all $0\le\eps<\widetilde\eps$, $\vec M\in\NN^K$, and $\vec k\in\vec K^{\infty}$,
			\begin{equation}
				\label{eq:thetaest}
				\frac{\vec a(\vec M,\vec k)Q\vec k \,e^\eps}{\vec k\cdot\vec p+\vec a(\vec M,\vec k)Q\vec k }\le \theta(\eps)<1.
			\end{equation}
			In addition, the function $\| {\bf v} \|_\eps$ is a norm on the space 
			$
			{\bf V}_\epsilon  := \{{\bf v} \, \Huge| \, \| {\bf v}  \|_\epsilon <\infty  \},
			$
			and the function $\vertiii{{\bf v}(t)}_\epsilon$ is a norm on the space $C^0_B([0,\infty), {\bf V}_\epsilon)$ of bounded continuous functions on $[0,\infty)$  taking values in ${\bf V}_\epsilon$. 
		\end{lemma}
		\begin{proof}
			These are standard results.
		\end{proof}
		
		\begin{theorem}
			\label{thm:convergence}
				Assume the conditions of Lemma~\ref{lem:elem}. Let ${\mathbf  u}^{(\infty)}$ and ${\mathbf  u}^{(M)}$  be the solutions of~\eqref{eq:master_inf} and~\eqref{eq:odesys-M-embed-infinity}, respectively. Then for $0<\eps<\widetilde\eps$,
				\begin{equation}
					\label{eq:limu}
					\altvertiii{\myvec u^{(\vec M)}-{\myvec u}^{(\infty)}}_\eps=O\left(\frac1{M}\right), \qquad M\to\infty.
				\end{equation}
			\end{theorem}
			
			\begin{proof}
				Since $\{u^{(M)}_{\vec k}\}$ are probabilities for $\vec{k}\in \vec K^{\vec M}$, see~\eqref{eqs:pode}, they are bounded between~0
				and~1.
				In addition, $\{u^{(M)}_{\vec k}\}$ are given by~\eqref{eq:u_n^(M)-for-n>M-het} for $\vec k\in \vec K^{\infty}\backslash \vec K^{\vec M}$,
				and so are also bounded between~0 and~1.
				Therefore, we have the uniform bound
				\begin{equation}
					\label{eq:ubnd-het}
					\vertiii{{\mathbf  u}^{(M)}}_0 \le 1.
				\end{equation}
				Subtracting \eqref{eq:master_inf} from \eqref{eq:odesys-M-embed-infinity} yields
				\begin{equation}
					\label{eq:diffeq}
					\begin{aligned}
						\tddt \left[u^{(\vec M)}_{\vec k}-u_{\vec k}^{(\infty)}\right]&
						=-[\vec k\cdot\vec p +\vec aQ\vec k](u^{(\vec  M)}_{\vec k}-u_{\vec k}^{(\infty)})
						\\&\quad+ \vec aD\left(u^{(\vec M)}_{\vec k+\vec e_1}-u_{\vec k+\vec e_1}^{(\infty)},\cdots,u^{(\vec M)}_{\vec k+\vec e_K}-u_{\vec k+\vec e_K}^{(\infty)}\right)Q\vec k 
						\\&\quad  +[\vec a-\vec a(\vec M,\vec k)]D\left(u^{(\vec M)}_{\vec k}-u^{(\vec M)}_{\vec k+\vec e_1},\cdots,
						u^{(\vec M)}_{\vec  k}-u^{(\vec M)}_{\vec k+\vec e_K}\right)Q\vec k.
					\end{aligned}
				\end{equation}
				Multiplying 
				by the integrating factor $e^{(\vec k\cdot\vec p +\vec a Q\vec k)t}$,
				integrating from zero to~$t$, and using
				$u^{(M)}_{\vec k}(0)=u_{\vec k}^{(\infty)}(0)$,
				yields
				\footnotesize
				\begin{align}
					\label{eq:diffinteq}
					&u^{(\vec M)}_{\vec k}(t)-u_{\vec k}^{(\infty)}(t)
					\\=&\int_0^t e^{-(\vec k\cdot\vec p +\vec a Q\vec k)(t-s)}\vec a D\left(u^{(\vec M)}_{\vec k+\vec e_1}(s)-u_{\vec k+\vec e_1}^{(\infty)}(s),\cdots,u^{(\vec M)}_{\vec k+\vec e_K}(s)-u_{\vec k+\vec e_K}^{(\infty)}(s)\right)Q\vec k \,ds \nonumber
					\\+&\int_0^t e^{-(\vec k\cdot\vec p +\vec a Q\vec k)(t-s)}[\vec a-\vec a(\vec M,\vec k)] D\left(u^{(\vec M)}_{\vec k}(s)-u^{(\vec M)}_{\vec k+\vec e_1}(s),\cdots,
					u^{(\vec M)}_{\vec  k}(s)-u^{(\vec M)}_{\vec k+\vec e_K}(s)\right)Q\vec k\,ds, \nonumber
				\end{align}
				\normalsize
				since $\myvec u^{(\vec M)}$ and $\myvec u^{(\infty)}$ have the same initial data.
				Taking the $\vertiii{\,}_\eps$ norm with $0<\eps<\widetilde \eps$ of both sides of~\eqref{eq:diffinteq}, estimating
				on the right, and using \eqref{eq:ubnd-het}  yields
				\small
				\begin{equation}
					\label{eq:diffest}
					\begin{aligned}
						&\altvertiii{\myvec u^{(\vec M)}-\myvec u^{(\infty)}}_\eps
						\\&\quad\le
						\sup_{t,\vec k} [\vec a Q\vec k] e^{-\eps n(\vec k)} \int_0^t e^{-(\vec k\cdot\vec p +\vec a Q\vec k)(t-s)}
						e^{\eps(1+ n(\vec k)) }\altvertiii{\myvec u^{(\vec M)}-\myvec u^{(\infty)}}_\eps\,ds
						\\&\qquad+\sup_{t,\vec k} 2\sum_{i=1}^{K} \left|\vec a(\vec M,\vec k)(i)-a_i\right|\max_{i,j}q_{i,j}   n(\vec k) e^{-\eps n(\vec k)}\int_0^t e^{-(\vec k\cdot\vec p +\vec a Q\vec k)(t-s)} \altvertiii{\myvec u^{(\vec M)}}_0\,ds
						\\&\quad=\sup_{t,\vec k} \frac{  [\vec a Q\vec k] e^{\eps}}{\vec k\cdot\vec p +\vec a Q\vec k}\altvertiii{\myvec u^{(M)}-\myvec u^{(\infty)}}_\eps
						(1-e^{-(\vec k\cdot\vec p +\vec a Q\vec k)t})
						\\&\qquad+\sup_{t,\vec k} \frac{2\sum_{i=1}^{K} \left|\vec a(\vec M,\vec k)(i)-a_i\right|\max_{i,j}q_{i,j}   n(\vec k)   e^{-\eps n(\vec k)}}{\vec k\cdot\vec p +\vec a Q\vec k}
						(1-e^{(\vec k\cdot\vec p +\vec a Q\vec k)t})\altvertiii{\myvec u^{(\vec M)}}_0
						\\&\quad\le \theta(\eps) \altvertiii{\myvec u^{(\vec M)}-\myvec u^{(\infty)}}_\eps
						+2\max_{i,j}q_{i,j}\frac{\sup_{\vec k} \sum_{i=1}^{K} \left|\vec a(\vec M,\vec k)(i)-a_i\right| e^{-\eps n(\vec k)}}{\min_j p_j}.
					\end{aligned}
				\end{equation}
				\normalsize
				Since $\eps<\widetilde\eps$, then $\theta(\eps)<1$, see~\eqref{eq:supeps}, and so~\eqref{eq:diffest} implies that
				\small
				\begin{equation}
					\label{eq:diffest2}
					\altvertiii{\myvec u^{(M)}-\myvec u^{(\infty)}}_\eps\le \frac{2\max_{i,j}q_{i,j}}{\left(\min_j
						p_j\right)(1-\theta(\eps))} \sup_{\vec k} \Big[\sum_{i=1}^{K} \left|\vec a(\vec M,\vec k)(i)-a_i\right| e^{-\eps n(\vec k)}\Big].
				\end{equation}
				\normalsize
				Relation~\eqref{eq:limu} follows from~\eqref{eq:diffest2}~and Lemma~\ref{lem:q-q^m=O(1/M)} below.
			\end{proof}	
			\begin{corollary}
				\label{cor:Niu}
				For any $\vec k\in\vec K^{\infty}$, $u_{\vec k}^{(\vec M)}(t)\to u_{\vec k}^{(\infty)}(t)$ uniformly in $t$ as $M \to \infty$.  Moreover, $u_{\vec k}^{(\vec M)}(t)- u_{\vec k}^{(\infty)}(t)  = O\left(\frac1{M}\right)$ as $M \to \infty$. 
			\end{corollary}
			\begin{proof}
				By~\eqref{eq:epsnorm}
				\begin{equation}
					\label{eq:|u_n|<|||u_n|||}
					\left| u_{\vec k}^{(\vec M)}(t) -  u_{\vec k}^{(\infty)}(t) \right| \le e^{\eps n(\vec k)}\vertiii{\mathbf  u^{(M)}-\mathbf  u^{(\infty)}}_\eps, \qquad \vec k\in \vec K^{\infty}.
				\end{equation}
				Therefore, the result follows from the rate of convergence estimate \eqref{eq:limu}.
			\end{proof}
			\begin{lemma}
				\label{lem:q-q^m=O(1/M)}
				Let $0<\eps<\widetilde\eps$. Then 
				\begin{equation}
					\label{eq:q-q^m=O(1/M)}
					\sup_{\vec k\in\vec K^{\infty}} \Big[\sum_{i=1}^{K} \left|\vec a(\vec M(M),\vec k)(i)-a_i\right| e^{-\eps n(\vec k)}\Big]=O\left(\frac1{M}\right), \qquad M\to\infty.
				\end{equation}
			\end{lemma}
			\begin{proof}
				By~\eqref{eq:addalp},
				\begin{align*}
					\sup_{\vec k\in \vec K^{\infty}}\Big[\sum_{i=1}^{K} &\left|\vec a(\vec M(M),\vec k)(i)-a_i\right| e^{-\eps n(\vec k)}\Big]=
					\sup_{\vec k\in\vec K^{\infty}} \Big[\sum_{i=1}^K \left|\frac{M_i-k_i}{M-1}-a_i\right| e^{-\eps n(\vec k)}\Big]\\&\leq\sup_{\vec k\in\vec K^{\infty}} \Big[\sum_{i=1}^K \left|\frac{M_i}{M-1}-a_i\right| e^{-\eps n(\vec k)}\Big]+ \sup_{\vec k\in\vec K^{\infty}} \frac{1}{M-1}\Big[\sum_{i=1}^{K} k_ie^{-\eps n(\vec k)}\Big]\\
					&\leq \sum_{i=1}^{K} \left|\frac{M_i}{M-1}-a_i\right| + \frac{1}{M-1}\sup_{x\geq 0} \Big[xe^{-\eps x}\Big]=O\left(\frac1{M}\right),
				\end{align*}
				where the last equality follows from~\eqref{eq:pop_frac_lim}.
			\end{proof}
			

			\section{Effect of heterogeneity}
			\label{sec:effect}
			Despite the general agreement that individuals are anything but homogeneous, only a few studies out of the sizeable literature on compartmental Bass models analyzed the qualitative effect of heterogeneity in compartmental Bass models. 
			Chaterjee and Eliashberg constructed a compartmental diffusion model which allowed for  heterogeneity in consumers' initial perceptions and price hurdles, and showed that heterogeneity can alter the qualitative behavior of aggregate adoption~\cite{Chatterjee-Eliashberg-90}. Bulte and Joshi divided the population into two groups: The influentials with $p = p_1$ and $q = q_1$, and the imitators with $p=0$ and $q = q_2$. Their numerical results revealed that heterogeneity in $p$ and $q$ can change the qualitative behavior of the diffusion~\cite{Bulte2007NewPD}.
			
			In this section, we analyze the qualitative effect of heterogeneity, within the framework of the compartmental Bass model. In most of the analysis, we will be considering a {\em milder heterogeneity} in $q$, in which $\{q_j\}_{j=1}^M$ are allowed to be heterogeneous,
			but each individual is equally influenced by any adopter. Formula~\eqref{eq:q_on_node} then yields
			\begin{equation}
				\label{eq:q_mild_het}
				q_{m,j}={q_{j}}, \qquad m\neq j,
			\end{equation}
			and~\eqref{eq:heterogeneous_Bass} becomes
			\begin{equation}
				\label{eq:q_mild_heterogeneous_Bass}
				\frac{d}{dt}f_k(t)=\left(a_k-f_k\right)\left(p_k+q_kf^{\rm heter}_{\rm Bass}\right), \qquad f_k(0)=0, \qquad k=1,\ldots,K.
			\end{equation}
			Thus, we compare the fraction of adopters\newline $f^{\rm het}(t;\{p_j\},\{q_{i,j}\},\{a_j\}):=f^{\rm heter}_{\rm Bass}(t;\{p_j\},\{q_{i,j}\},\{a_j\})$ or $f^{\rm het}(t;\{p_j\},\{q_j\},\{a_j\}):=f^{\rm heter}_{\rm Bass}(t;\{p_j\},\{q_j\},\{a_j\})$ in the compartmental model~\eqref{eq:heterogeneous_Bass}, or~\eqref{eq:q_mild_heterogeneous_Bass}, respectively, with that in the corresponding homogeneous case $f^{\rm hom}:=f_{\rm Bass}(t;\bar{p},\bar{q})$, where
			\begin{equation}
				\label{eq:p_bar_q_bar}
				\bar{p}:=\sum_{k=1}^{K}a_kp_k,\qquad \bar{q}=\sum_{k=1}^{K}a_kq_k=\sum_{k=1}^{K}a_k\sum_{m=1}^{K}a_mq_{m,k}.
			\end{equation}
			We do that using the {\em Bass Inequality Lemma}:
			\begin{lemma} 
				\label{lem:Bass_inequality}
				Let $f^{\rm het}(t)$ satisfy the Bass inequality
				\begin{equation*}
					\frac{d}{dt}f^{\rm het}(t)<(1-f^{\rm het})(p+qf^{\rm het}),\qquad f^{\rm het}(0)=0.
				\end{equation*}
				Then
				\begin{equation*}
					f^{\rm het}(t)<f_{\rm Bass}(t;p,q), \qquad 0<t<\infty.
				\end{equation*}
			\end{lemma}
			\begin{proof}
				This result was proved in~\cite[Appendix A]{OR-10}. It is also a special case of Lemma~\ref{lem:dif_ineq} below.~
			\end{proof}
			To apply Lemma~\ref{lem:Bass_inequality}, let
			\begin{equation}
				\label{eq:y_def}
				y:=\left(1-f^{\rm het}\right)\left(p+qf^{\rm het}\right)-\frac{d}{dt}{f^{\rm het}}.
			\end{equation}
			Then
			\begin{equation*}
				\label{eq:diff_with_y}
				\frac{d}{dt}f^{\rm het}(t)=(1-{f^{\rm het}})(p+q{f^{\rm het}})-y(t), \qquad 0<t<\infty.
			\end{equation*}
			Hence, if we can show that $y(t)>0$ for $0<t<\infty$, then by Lemma~\ref{lem:Bass_inequality}, heterogeneity slows down the adoption. 
			
			In Theorem~\ref{lem:pos_cor}, we will show that $y>0$ using the following auxiliary result:
			\begin{lemma}\label{lem:ODElem}
				\label{lem:dif_ineq}
				Let 
				$$
				\frac{d}{dt}f_1(t)<H_2(f_1),\qquad \frac{d}{dt}f_2(t)=H_2(f_2), \qquad  0<t<\infty, \qquad f_1(0)=f_2(0).
				$$ 
				Assume that $H_2$ is Lipschitz continuous. Then 
				$$
				f_1(t)<f_2(t), \qquad 0<t<\infty.
				$$
			\end{lemma}
			\begin{proof}
				See Section~\ref{sec:ODElem}.
			\end{proof}
			
			We now use Lemmas~\ref{lem:Bass_inequality}~and~\ref{lem:dif_ineq} to analyze the case where the heterogeneity in $q$ is
			mild, and $\{p_k\}$ and $\{q_k\}$ are positively  monotonically related, i.e., $p_i \le p_j \iff q_i \le p_q$. Without loss of generality, we can assume that
			\begin{subequations}
				\label{eqs:cor_conditions}
				\begin{equation}
					\label{eq:pq_correlated}
					p_1\leq p_2 \leq \dots \leq p_K, \qquad q_1\leq q_2 \leq \dots \leq q_K.
				\end{equation}  
				We assume that the network is not homogeneous, i.e., that
				\begin{equation}
					\label{eq:het_condition}
					\exists i,j \ {\rm such\ that}\ p_i>p_j \ {\rm or }\ q_i>q_j.
				\end{equation}
			\end{subequations}
			\begin{theorem}
				\label{lem:pos_cor}
				Let $f^{\rm het}$ and $f^{\rm hom}$ denote the fraction of adopters in the mildly heterogeneous compartmental Bass model~\eqref{eq:q_mild_heterogeneous_Bass} and in the corresponding homogeneous model~\eqref{eq:homogeneous_Bass} with $\bar{p}:=\sum_{k=1}^{K}a_kp_k$ and $\bar{q}:=\sum_{k=1}^{K}a_kq_k$, respectively. Assume that~\eqref{eqs:cor_conditions} holds. Then
				\begin{equation*}
					f^{\rm het}(t; \{p_k\}_{k=1}^K,\{q_k\}_{k=1}^K,\{a_k\}_{k=1}^K)<f^{\rm hom}(t; \bar{p},\bar{q}),\qquad 0<t<\infty.
				\end{equation*}
			\end{theorem}
			\begin{proof}
				Let $y$ be given by~\eqref{eq:y_def}. Then by~\eqref{eq:heterogeneous_Bass},
				\begin{align*}
					\nonumber
					y &= \left(1-f^{\rm het}\right)\left(\bar{p}+\bar{q}f^{\rm het}\right)-\sum_{k=1}^{K}\frac{d}{dt}f_k\\&=\sum_{k=1}^{K}\left(a_k-f_k\right)\left(\bar{p}+\bar{q}f^{\rm het}\right)-\sum_{k=1}^{K}\left(a_k-f_k\right)\left(p_k+q_kf^{\rm het}\right)\nonumber\\\nonumber
					&=\sum_{k=1}^{K}\left(a_k-f_k\right)\left(\bar{p}-p_k+f^{\rm het}\left(\bar{q}-q_k\right)\right)=\sum_{k=1}^{K}f_k(p_k-\bar{p})+f^{\rm het} \sum_{k=1}^{K}f_k(q_k-\bar{q}) .
				\end{align*}
				Thus, 
				\begin{equation}
					\label{eq:y_correlated}
					y =\sum_{k=1}^{K}f_k p_k-\bar{p}f^{\rm het}
					+f^{\rm het}\left( \sum_{k=1}^{K}f_k q_k-\bar{q} f^{\rm het}\right). \end{equation}			
				We claim that 
				\begin{equation}
					\label{eq:claim_correlated_y}
					y(t)>0, \qquad 0<t<\infty.
				\end{equation}
				Therefore, the result follows from Lemma~\ref{lem:Bass_inequality}. 
				
				To prove~\eqref{eq:claim_correlated_y}, let 
				$$
				\tilde{f}_k(t):=\frac{f_k}{a_k}
				$$
				denote the fraction of adopters within group~$k$, and let
				$$
				H_k(\tilde{f}_k):=(1-\tilde{f}_k)(p_k+q_kf^{\rm het}).
				$$ 
				Then~\eqref{eq:q_mild_heterogeneous_Bass} reads
				$$
				\frac{d}{dt}\tilde{f}_k(t)=H_k(\tilde{f}_k), \qquad \tilde{f}_k(0)=0,\qquad k=1,\ldots,K.
				$$
				Hence,
				$$
				\begin{aligned}
					H_i(\tilde{f}_j)-\frac{d}{dt}\tilde{f}_j&=\left(1-\tilde{f}_j\right)(p_i+q_if^{\rm het})-\left(1-\tilde{f}_j\right)(p_j+q_jf^{\rm het})\\&=\left(1-\tilde{f}_j\right)\left(\left(p_i-p_j\right)+f^{\rm het}\left(q_i-q_j\right)\right).
				\end{aligned}
				$$
				
				Let $i>j$ be for which~\eqref{eq:het_condition} holds.
				Since  $\tilde{f}_j, f^{\rm het}>0$ for  $t>0$, 
				then
				$$
				\left(1-\tilde{f}_j\right)\left(\left(p_i-p_j\right)+f^{\rm het}\left(q_i-q_j\right)\right)>0,
				$$ 
				and so
				$$
				\frac{d}{dt}\tilde{f}_j<H_i(\tilde{f}_j), \qquad 0<t<\infty.
				$$
				Therefore, by Lemma~\ref{lem:dif_ineq},
				\begin{equation}
					\label{eqs:f_ineqs}
					\tilde{f}_j(t)<\tilde{f}_i(t),\qquad 0<t<\infty.
				\end{equation}
				Hence,
				\begin{equation*}
					0<(\tilde{f}_i-\tilde{f}_j)\left(p_i-p_j+f^{\rm het}\left(q_i-q_j\right)\right)=\left(\frac{f_i}{a_i}-\frac{f_j}{a_j}\right)\left(p_i-p_j+f^{\rm het}\left(q_i-q_j\right)\right)
				\end{equation*}
				for $0<t<\infty$. This inequality becomes an equality whenever condition~\eqref{eq:het_condition} does not hold, so that $p_i=p_j$ and 
				$q_i=q_j$. Hence,
				\begin{align*}
					0&<\sum_{i=1}^{K}\sum_{j=1}^{K}a_ia_j\left(\frac{f_i}{a_i}-\frac{f_j}{a_j}\right)\left(p_i-p_j+f^{\rm het}\left(q_i-q_j\right)\right)
					\\&=~~\sum_{i=1}^{K}f_ip_i\sum_{j=1}^{K}a_j-\sum_{i=1}^{K}f_i\sum_{j=1}^{K}a_jp_j-\sum_{i=1}^{K}a_ip_i\sum_{j=1}^{K}f_j+\sum_{i=1}^{K}a_i\sum_{j=1}^{K}f_jp_j
					\\&~~+f^{\rm het}\left(
					\sum_{i=1}^{K}f_iq_i\sum_{j=1}^{K}a_j-\sum_{i=1}^{K}f_i\sum_{j=1}^{K}a_jq_j-\sum_{i=1}^{K}a_iq_i\sum_{j=1}^{K}f_j+\sum_{i=1}^{K}a_i\sum_{j=1}^{K}f_jq_j\right)
					\\&=2\left(\sum_{i=1}^{K}f_ip_i-f^{\rm het}\bar{p}+f^{\rm het}\left(\sum_{i=1}^{K}f_iq_i-f^{\rm het}\bar{q}\right)\right)=2y,
				\end{align*}
				see~\eqref{eq:y_correlated}. 
				Therefore, we proved~\eqref{eq:claim_correlated_y}.
			\end{proof}
			This result was also derived by Fibich and Golan in~\cite{fibich2021diffusion}, but their proof is much more complicated, since they analyzed the discrete Bass model directly.
			\begin{corollary}
				\label{cor:correlated}
				Let $\{p_k\}$ and $\{q_k\}$ be monotonically increasing in~$k$, see~\eqref{eq:pq_correlated}. Then the adoption level within each group in the compartmental Bass model~\eqref{eq:q_mild_heterogeneous_Bass} is monotonically increasing in~$k$, i.e.,
				$$
				\tilde{f}_1(t)\leq \tilde{f}_2(t)\leq\ldots\leq \tilde{f}_K(t), \qquad 0<t<\infty.
				$$
			\end{corollary}
			\begin{proof}
				This follows from relation~\eqref{eqs:f_ineqs}.
			\end{proof}
			This result is intuitive. Indeed, let $1 \le k < m \le K$. Then nonadopters from group~$k$ experience both weaker external influences $p_k \le  p_m$ and weaker internal influences $q_kf \le  q_m f$
			than nonadopters from group~$m$,
			see~\eqref{eq:q_mild_heterogeneous_Bass}. Therefore, they are slower to adopt.

			In light of Theorem~\ref{lem:pos_cor}, it is natural to ask whether $f^{\rm het}<f^{\rm hom}$ even if $\{p_j\}$ and
			$\{q_j\}$ are not positively  monotonically related. The following example shows that this is not always the case.
			
			\begin{example}
				Let $K=2$, $a_1=a_2=\frac{1}{2}$, $p_1=0$, $p_2=2p$, $q_{1,1}=0$, $q_{1,2}=0$, $q_{2,1}=4q$, and $q_{2,2}=0$ (see Figure~\ref{fig:het_greater_than_hom}A). Then~\eqref{eq:heterogeneous_Bass} reads
				\begin{equation}
					\label{eq:het_faster_sys}
					\begin{cases}
						f_1'(t) = \left(\frac{1}{2}-f_1\right)4qf_2,& \qquad f_1(0)=0,\\
						f_2'(t) = \left(\frac{1}{2}-f_2\right)2p,& \qquad f_2(0)=0.
					\end{cases}
				\end{equation}
				Since $q_{2,1}\neq q_{1,1}$, the network is not mildly-heterogeneous in $q$. Since $q_1=2q$ and $q_2=0$, see~\eqref{eq:max_q_compartmental}, the corresponding homogeneous model is~\eqref{eq:homogeneous_Bass} with $\bar{p}:=p$ and $\bar{q}:=\frac{1}{2}\left(q_1+q_2\right)=q$, see~\eqref{eq:p_bar_q_bar}.
				\begin{lemma}
					\label{lem:het_faster_hom}
					Let $f^{\rm het}(t)$ and $f^{\rm hom}$ denote the fraction of adopters in the heterogeneous model~\eqref{eq:het_faster_sys} and in the corresponding homogeneous model~\eqref{eq:homogeneous_Bass} with $\bar{p}=p$ and $\bar{q}=q$, respectively. If $q>p>0$, then
					\begin{equation}
						\label{eq:het_greater_hom}
						f^{\rm het}(t)>f^{\rm hom}(t), \qquad 0<t\ll 1.
					\end{equation}
				\end{lemma}
				\begin{proof}
					It is easy to verify that $\frac{d}{dt}f^{\rm het}(0)=\frac{d}{dt}f^{\rm hom}(0)=p$, $\frac{d^2}{dt^2}f^{\rm het}(0)=2p(q-p)$, and $\frac{d^2}{dt^2}f^{\rm hom}(0)=p(q-p)$. Hence, when $q>p>0$, $\frac{d^2}{dt^2}f^{\rm het}(0)>\frac{d^2}{dt^2}f^{\rm hom}(0)$, and so~\eqref{eq:het_greater_hom} holds.
				\end{proof}
				\noindent Figure~\ref{fig:het_greater_than_hom}B shows that indeed,
				when $q>p$, then initially $f^{\rm het}(t)>f^{\rm hom}(t)$. This dominance flips, however, as $t\to\infty$.
			\end{example}
			
			\begin{figure}[ht!]
				\centering
				\begin{minipage}{.3\textwidth}
					\centering
					\begin{tikzpicture}
						\draw (5,0) rectangle (.4,3);
						
						\draw (1.4,1.25) circle(0.6cm) node {$p_1=0$};
						\draw (4, 1.25) circle(0.6cm) node {$p_2=2p$};
						
						\draw [->, thick] (3.35,1.25)--(2.05,1.25);
						
						\draw (2.7,1.75) node{4q};
						
						\draw (4.5,2.25) node{A};
					\end{tikzpicture}  
				\end{minipage}\hspace{15mm}
				\begin{minipage}{.3\textwidth}
					\begin{center}
						\scalebox{1}{\includegraphics{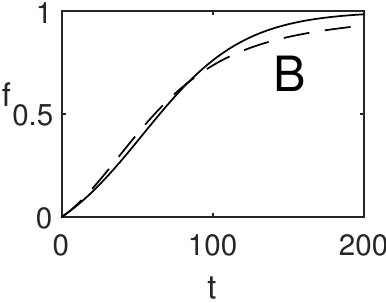}}
					\end{center}
				\end{minipage}
				\caption{Aggregate adoption in the heterogeneous (dashes) network given by~\eqref{eq:het_faster_sys} is locally greater than in the corresponding homogeneous (solid) network~\eqref{eq:homogeneous_Bass} for~$p=0.2q$.}
				\label{fig:het_greater_than_hom}
			\end{figure}
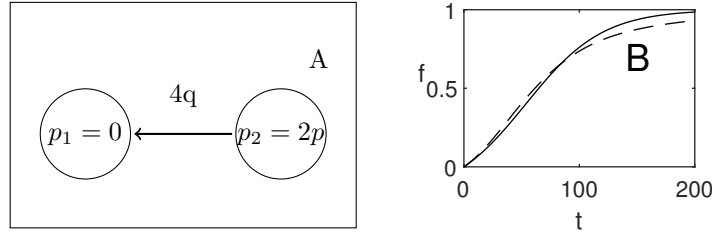
			Thus, $f^{\rm het}$ can be larger than $f^{\rm hom}$ if the heterogeneity in $q$ is not mild. But, can we have that $f^{\rm het}>f^{\rm hom}$ when the heterogeneity in $q$ is mild?
			
			We now proceed to analyze a case of a mild heterogeneity in $q$, in which $\{p_j\}$ and $\{q_j\}$ are negatively
			monotonically related.
			\begin{example}
				Let $K=2$, $a_1=a_2=\frac{1}{2}$, $p_1=0$, $q_1=2q$, $p_2=2p$, and $q_2=0$. Then the compartmental model~\eqref{eq:q_mild_heterogeneous_Bass} reads
				\begin{equation}
					\label{eq:deterministic_system2}
					\begin{cases}
						f_1'(t) = \left(\frac{1}{2}-f_1\right)2q(f_1+f_2),& \qquad f_1(0)=0,\\
						f_2'(t) = \left(\frac{1}{2}-f_2\right)2p,& \qquad f_2(0)=0.
					\end{cases}
				\end{equation}
				
				\begin{lemma}
					\label{lem:het<hom_pq}
					Let $f^{\rm het}(t)$ and $f^{\rm hom}(t)$ denote the fraction of adopters in the heterogeneous model~\eqref{eq:deterministic_system2} and in the corresponding homogeneous model~\eqref{eq:homogeneous_Bass} with $\bar{p}=p$ and $\bar{q}=q$, respectively. If $p\geq q$, then $f^{\rm het}(t)<f^{\rm hom}(t)$ for $0<t<\infty$.
				\end{lemma}
				\begin{proof}
					By~\eqref{eq:deterministic_system2}, $f^{\rm het}:=f_1+f_2$ satisfies
					\begin{equation}
						\label{eq:deterministic_pq}
						\frac{d}{dt}f^{\rm het}(t)=\left(1-f^{\rm het}\right)\left(p+qf^{\rm het}\right)-y, \qquad y:=\left(f_1-f_2\right)\left(qf^{\rm het}-p\right).
					\end{equation}
					Since $p\geq q$, then $qf^{\rm het}-p<0$ and so by \eqref{eq:deterministic_system2}, $f_2>f_1$ for $0<t<\infty$. This implies that $y(t)<0$. Therefore, by Lemma~\ref{lem:Bass_inequality}, $f^{\rm het}(t)<f^{\rm hom}(t)$.
				\end{proof}
			\end{example}

			\begin{figure}[ht!]
				\begin{center}
					\scalebox{1}{\includegraphics{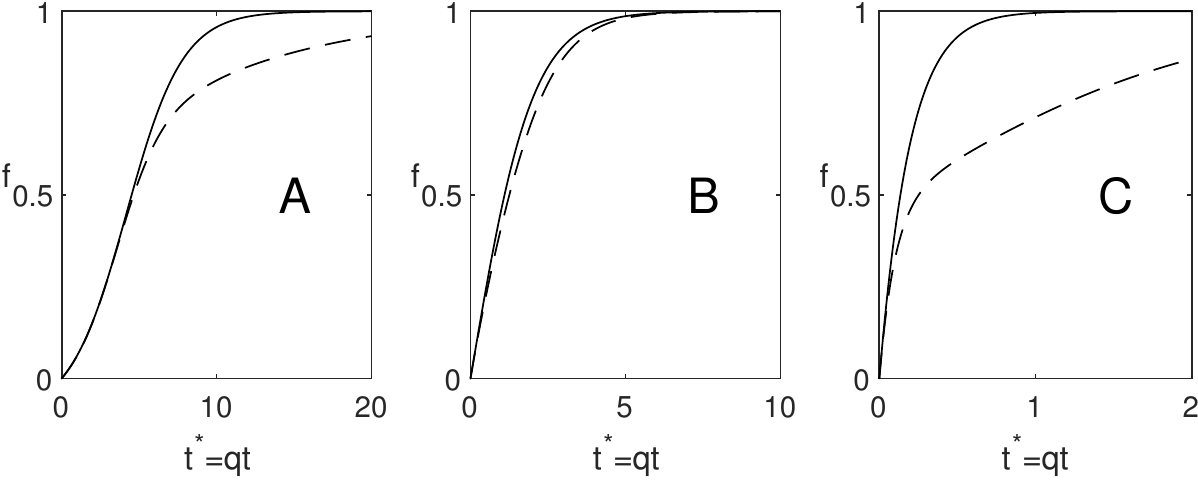}}
					\caption{Aggregate adoption in the heterogeneous (dashes) compartmental network~\eqref{eq:deterministic_system2} is slower than in the corresponding homogeneous (solid) network~\eqref{eq:homogeneous_Bass}. A)~$p=0.1q$. B)~$p=q$. C)~$p=10q$.}
					\label{fig:heter_vs_homog_fc_M=10000}
				\end{center}
			\end{figure}
			
			
			Numerical simulations of~\eqref{eq:deterministic_system2} confirm that $f^{\rm het}(t)<f^{\rm hom}(t)$ for $p\geq q$  (Figure~\ref{fig:heter_vs_homog_fc_M=10000}B-C). Moreover, they show that $f^{\rm het}(t)<f^{\rm hom}(t)$ for $p<q$ as well (Figure~\ref{fig:heter_vs_homog_fc_M=10000}A). 
			This suggests, therefore, that perhaps in the mildly-heterogeneous case~\eqref{eq:q_mild_heterogeneous_Bass}, $f^{\rm het}<f^{\rm hom}$ for any $\{p_k\}$ and  $\{q_k\}$, and not just when they are  positively  monotonically related. Whether this is the case, however, is currently open. A proof of Lemma~\ref{lem:het<hom_pq} for $p<q$ is also open.
			
			
			
			\section{Discussion}
			\label{sec:discussion}
			
			The problem of proving that solutions of a system of master equations converge, as the number of individuals tends to infinity, to solutions
			of a compartmental model 
			has been widely studied in many disciplines. 
			The case most relevant for the Bass model for the 
			diffusion of new products is that of the
			susceptible infected (SI) model in epidemiology~\cite{MR3116158}, 
			in which susceptible individuals correspond to nonadopters and infectives
			correspond to adopters.
			Some epidemiological models also include recovered individuals
			which correspond to ``non-contagious'' adopters, 
			a category not included in the Bass model, but included in the Bass-SIR model~\cite{Bass-SIR-model-16, Bass-SIR-analysis-17}.
			
			The approach developed here for proving that convergence consists of three steps.
			\begin{enumerate}
				\item
				The full set \eqref{eqs:masterHeterofc} of master equations
				for the probabilities of the states of all subsets of nodes
				is reduced to equations for a smaller set of variables, which is sufficient in the sense that the expected fraction of adopters can be written in
				terms of them, and is closed in the sense that the equations for those variables involve only those variables. 
				The reduced equations for the homogeneous complete network 
				are~\eqref{eqs:masterfc}; for the homogeneous circle network are~\eqref{eq:mastereqs1d},
				and for the heterogeneous complete network
				are~\eqref{eqs:pode}.
				\item
				The reduced equations, which are a finite system of ODEs, are embedded into an infinite system of ODEs and the convergence of that system to the infinite limit system is proven.
				\item
				An ansatz, such as~\eqref{eq:inf_sol_hom} for the homogeneous complete network, \eqref{eq:substitution}~for the homogeneous
				circle network, and~\eqref{eqs:subbed_equation} for the heterogeneous complete network, provides an exact closure of the infinite limit system, which reduces that system to the compartmental model.\footnote{The infinite limit system is simpler than the large but finite system of reduced master equations that tends to it, because the
					infinite system no longer contains the number of individuals as a parameter. Hence it may be easier to obtain an exact moment closure for the limit system.}
			\end{enumerate}

			Regarding the analogous problem of proving convergence of epidemiological models to their mean field limits, there are several works with a variety of methods. To the best of our knowledge, all of these studies only considered  homogeneous complete networks, and it is not clear whether their methods can be extended to other types of networks.\footnote{
				Simon and Kiss~\cite{MR3116158} also proved the convergence for
				n-random networks, but used an approximation for the number of SI pairs to close the system, which was not rigoroudsly justified.} 
			In contrast, our methodology can be applied to various types of networks,
			with only minor modifications. 
			
			The most similar work to ours is~\cite{MR3116158}, in which Simon and Kiss used an ODE approach to prove the convergence of the SIS model on complete networks to its compartmental limit. The key difference between their approach and ours is that our starting point 
			are the master equations for the probabilities of the states of all subsets of nodes (a {\em bottom-up} approach), whereas Simon and Kiss started from the
			master equations for the probability of having $k$ adopters in the system at time~$t$ (a {\em top-down} approach). It is not clear, however, whether a top-down approach can be extended to 
			other types of netowrks beyond homogeneous complete networks, without introdcing a closure which is not rogorously justified. Other approaches for proving convergence of epidemiological models on homogeneous complete networks include a PDE approach~\cite{Diekmann2000, MR3116158, MR2776920}, a stochastic approach~\cite{Ethier2005, Kurtz-70, Kurtz-71, MR3116158}, and  an elementary ODE approach that only requires a finite number of ODEs~\cite{Armbruster-17}. 
		
		We now discuss the general problem of deducing convergence of solutions of finite systems of ODEs to those of infinite
		systems of ODEs. The formal limit system as $M\to\infty$ of a system of $M$ ODEs and initial conditions is the infinite system
		obtained by taking the limit of the equation and initial condition for each $u_k$ separately, with all components $u_j$
		considered to be fixed.
		
		Not surprisingly, if the component differential equations or initial conditions do not converge
		formally, then the solutions of the finite system may not converge. For example, if $\tddt u_k=M-u_k$ for $1\le k\le M$ and
		$u_k(0)=0$, then for $t>0$ the solutions $u_k(t)=M(1-e^{-t})$ of the finite system tend to infinity as $M\to\infty$.  Similarly, if the initial value of $u_M$
		is chosen sufficiently large then solutions of the finite system may tend to infinity even though the system converges
		formally to a limit system. For example, let
		\begin{equation}\label{eq:biguM}\tddt u_k=u_{k+1} \quad\text{and}\quad u_k(0)=0\end{equation}
		for $1\le k\le M-1$, and let $\tddt u_M=0$
		with $u_M(0)=(M!)^2$. For $M>k$ the ODE and initial condition for $u_k$ are \eqref{eq:biguM} so the system tends formally to
		the infinite system in which \eqref{eq:biguM} holds for all $k$. However, $u_{M-1}(t)=(M!)^2t$, $u_{M-2}(t)=(M!)^2\frac{t^2}2$, and by induction
		$u_{M-j}(t)=(M!)^2 \frac{t^j}{j!}$, so setting $k=M-j$ yields $u_k(t)=\frac{(M!)^2t^M}{(M-k)!}$, which for $t>0$ tends to
		infinity with $M$.
		
		Less obviously, {\em the formal convergence of the component equations and initial conditions does not suffice in general to yield
			convergence of the solutions of the finite system to a specified soluton of the infinite system obtained as the formal limit,
			even when the initial data of the finite systems are bounded uniformly in~$M$}. This shows that the result of
		Theorem~\ref{thm:convergence} is nontrivial.                        
		To demonstrate this phenomenon, consider the model system
		\begin{equation}\label{eq:2krule}
			\tddt u_{k}=a(k)(u_{k+1}-u_k), \qquad u_k(0)=1
		\end{equation}
		for $1\le k\le M$, where the parameters $a(k)$ and the value used to replace the non-existent component $u_{M+1}$ appearing in the
		ODE for $u_M$ remain to be specified. For later use, note that
		solving \eqref{eq:2krule} with $u_{k+1}$ considered to be a known function yields
		\begin{equation}
			\label{eq:2krulesol}
			u_k(t)=e^{-a(k)t}+a(k)e^{-a(k)t}\int_0^t u_{k+1}(s)e^{a(k)s}\,ds.
		\end{equation}
		For any fixed $k$ the rule \eqref{eq:2krule} applies for all $M\ge k+1$, so the formal limit system consists of
		\eqref{eq:2krule} for all $1\le k<\infty$, no matter what value is chosen for $u_{M+1}$. For any choice
		of the parameters $a(k)$ this limit system has a solution
		\begin{equation}\label{eq:expectedsol}
			u_k(t)\equiv1 \quad\text{for all $k$.}
		\end{equation}
		However, we will show that for certain parameters $a(k)$ and values used to replace $u_{M+1}$ 
		the solutions of the finite systems do not converge to the specified solution \eqref{eq:expectedsol} of the limit system. Our
		method does not show whether the solutions converge  to a different solution of the limit system or do not converge.
		
		Two plausible replacements for $u_{M+1}$ are the value
		\begin{equation}\label{eq:uhigh1} u_{M+1}\equiv 1\end{equation}
		equal to the initial value of all components, and the value
		\begin{equation}\label{eq:uhigh0} u_{M+1}\equiv0\end{equation}
		which corresponds to simply deleting the non-existent variable $u_{M+1}$ from the ODE for $u_M$. When \eqref{eq:uhigh1} is
		used then \eqref{eq:expectedsol} holds for every fixed $k$ even when $M$ is finite, no matter how the parameters $a(k)$ are chosen,
		because the right sides of the ODEs in \eqref{eq:2krule} are then identically zero.
		
		When $a(k)=1$ for all $k$ and $u_M$ is determined by \eqref{eq:uhigh0} then the solutions $u_k$ of the finite system
		are not exactly equal to \eqref{eq:expectedsol}, but they converge to that value as $M\to\infty$. To see this, solve the ODE for $u_M$ to obtain $u_M(t)=e^{-t}$, then
		substitute this result into \eqref{eq:2krulesol} with $k=M-1$ to obtain $u_{M-1}=(1+t)e^{-t}$. By induction 
		we obtain from \eqref{eq:2krulesol} that $u_{M-j}(t)=P_{j}(t)e^{-t}$, where $P_j(t)$ is the Taylor polynomial approximation of $e^t$ whose highest term is
		$\frac{t^j}{j!}$. Setting $j=M-k$ yields $u_k(t)=P_{M-k}(t)e^{-t}$ for $1\le k\le M$. Since $P_{M-k}(t)$ converges to $e^t$ as
		$M\to\infty$, $u_k(t)$ tends to \eqref{eq:expectedsol} in that limit.
		
		When $a(k)=3^k$ and $u_{M+1}$ is given by \eqref{eq:uhigh0} then the solutions $u_k$ can no longer be calculated
		explicitly, but we can obtain upper bounds that imply that the solutions do not converge to
		\eqref{eq:expectedsol}. The solution formula~\eqref{eq:2krulesol} yields $u_M=e^{-3^Mt}$, and substituting this into \eqref{eq:2krulesol} for
		$k=M-1$ yields
		$u_{M-1}(t)=e^{-3^{M-1}t}+\frac{3^{M-1}}{3^M-3^{M-1}}(e^{-3^{M-1}t}-e^{-3^Mt})\le 2 e^{-3^{M-1}t}$.
		Assuming by induction that $u_{k+1}\le 2 e^{-3^{k+1}t}$ and substituting that estimate into \eqref{eq:2krulesol} yields
		$u_k(t)\le e^{-3^kt}+\frac{2\; 3^k}{3^{k+1}-3^k}(e^{-3^kt}-e^{-3^{k+1}t})\le 2 e^{-3^kt}$, which confirms that the estimate
		$u_k(t)\le 2 e^{-3^kt}$ holds for all $k\le M$. Since $2 e^{-3^kt}<1$ for $t>\frac{\ln 2}{3^k}$, $u_k(t)$ does not converge to
		$1$ as $M\to\infty$.

		\section{Proof of Lemma~\ref{lem:ODElem}}\label{sec:ODElem}
		\begin{proof}
			Let $g:=f_1-f_2$. Define
			\begin{equation}\label{eq:Edef}
				E(s)\eqdef \begin{cases} s^2& s>0\\0&s\le0.\end{cases}
			\end{equation}
			Then $E(g(t))$ is differentiable and
			\begin{equation}
				\label{eq:Egprime}
				\begin{aligned}
					\tfrac{d\hfil}{dt} E(g(t))&\leq\begin{cases} 2g(t)[H_2(f_1)-H_2(f_2)]&g(t)>0\\0&g(t)\le0\end{cases}
					\\&\le \begin{cases} 2g(t) L|g(t)|&g(t)>0\\0&g(t)\le0\end{cases}\quad=2LE(g(t)).
				\end{aligned}
			\end{equation}
			Hence $\frac{d\hfil}{dt} e^{-2Lt}E(g(t))\le0$, which can be integrated to yield $e^{-2Lt} E(g(t))\le E(g(0))=0$ for $t>0$, which by the definition of
			$E$ implies that
			\begin{equation}\label{eq:t>0gle0}
				\text{$g(t)\le0$ for $t>0$.}
			\end{equation}
			Now suppose that  there exists a positive $t_*$ at which $g(t_*)\ge0$.  By~\eqref{eq:t>0gle0},  $g(t_*)=0$. Hence, $t_*$
			is a local maximum of $g$, and so $g'(t_*)=0$. However, for $t>0$, 
			$g'(t)<H_2(f_2+g)-H_2(f_2)$, so $g'<0$ whenever $g=0$ at a positive time. This contradiction shows that no such $t_*$ exists, and hence
			$f_1(t)-f_2(t)=g(t)<0$ for all $t>0$.
		\end{proof}

			\noindent {\em Email address: \href{mailto:fibich@tau.ac.il}{fibich@tau.ac.il}}\newline
			{\em Email address: \href{mailto:amitgolan33@gmail.com}{amitgolan33@gmail.com}}\newline
			{\em Email address: \href{mailto:schochet@tauex.tau.ac.il}{schochet@tauex.tau.ac.il}}
		\end{document}